\documentclass[12pt,a4paper]{amsart}
\usepackage{amsmath,amsthm,amsfonts,relsize,units,amssymb,setspace}

\DeclareRobustCommand{\rchi}{{\mathpalette\irchi\relax}}
\newcommand{\irchi}[2]{\raisebox{\depth}{$#1\chi$}} 

\usepackage[colorlinks, citecolor=blue]{hyperref}

\DeclareMathOperator\ad{ad}
\DeclareMathOperator\Ad{Ad}

\DeclareMathOperator\ran{Ran}

\DeclareMathOperator\sinc{\mathrm{sinc}}

\DeclareMathOperator\dist{\mathrm{dist}}

\DeclareMathOperator\A{\mathcal{A}}
\DeclareMathOperator\C{\mathcal{C}}
\DeclareMathOperator\D{\mathcal{D}}
\DeclareMathOperator\U{\mathcal{U}}

\DeclareMathOperator\h{\mathcal{H}}

\DeclareMathOperator\Exp{Exp}
\DeclareMathOperator\si{{\mathfrak s}}
\DeclareMathOperator\sinhc{\mathrm{sinhc}}

\begin{document}

\newtheorem*{theorem*}{Theorem}
\newtheorem{teo}{Theorem}[section]
\newtheorem{prop}[teo]{Proposition}
\newtheorem{lem}[teo]{Lemma}
\newtheorem{coro}[teo]{Corollary}
\theoremstyle{definition}
\newtheorem{defi}[teo]{Definition}
\newtheorem{rem}[teo]{Remark}
\newtheorem{ejem}[teo]{Example}
\newtheorem{problem}[teo]{Problem}
\newtheorem{conj}[teo]{Conjecture}

\markboth{}{}

\makeatletter

\title[Conjugate points in the Grassmannian]{\vspace*{-2cm}Conjugate points in the Grassmann manifold of a $C^*$-algebra}
\date{}
\author{Esteban Andruchow}
\address{E. Andruchow: Instituto Argentino de Matem\'atica ``Alberto P. Calder\'on'' (CONICET), Saavedra 15 3º piso (1083) CABA, Argentina \& Instituto de Ciencias,  Universidad Nacional de Gral. Sarmiento}
\email{eandruch@ungs.edua.ar}

\author{Gabriel Larotonda}
\address{G. Larotonda:  Departamento de Matem\'atica, Facultad de Cs. Exactas y Naturales, Universidad de Buenos Aires. Avenida Cantilo s/n, Ciudad Universitaria, Pabell\'on I  (1428) CABA, Argentina  \& Instituto Argentino de Matem\'atica ``Alberto P. Calder\'on'' (CONICET), Saavedra 15 3º piso (1083) CABA, Argentina}
\email{glaroton@dm.uba.ar}

\author{L\'azaro Recht}
\address{L. Recht: Instituto Argentino de Matem\'atica  ``Alberto P. Calder\'on'' (CONICET), Saavedra 15 3º piso (1083) CABA, Argentina}
\email{lrecht@usb.ve}


\makeatother

\begin{abstract}Let $Gr$ be a component of the Grassmann manifold of a $C^*$-algebra, presented as the unitary orbit of  a given orthogonal projection $Gr=Gr(P)$. There are several natural connections in this manifold, and we first show that they all agree (in the presence of a finite trace in $\mathcal A$, when we give $Gr$ the Riemannian metric induced by the Killing form, this is the Levi-Civita connection of the metric). We study the cut locus of $P\in Gr$ for the spectral rectifiable distance, and also the conjugate tangent locus of $P\in Gr$ along a geodesic. Furthermore, for each tangent vector $V$ at $P$, we compute the kernel of the differential of the exponential map of the connection. We exhibit examples where points that are tangent conjugate in the classical  setting, fail to be conjugate: in some cases they are not monoconjugate but epinconjugate, and in other cases they are not conjugate at all.
\end{abstract}

\maketitle

\setlength{\parindent}{0cm} 

\thispagestyle{empty}
\tableofcontents

\newpage 

\section{Introduction}

The Grassmann manifold $Gr_k(n)$ of $k$-planes in $n$-dimensional (real or complex) space is a source of explicit computations in the classical Riemannian geometry. There are several ways to introduce a Riemannian metric which are equivalent: geometrically, this can be done by means of the principal angles $\{\theta_1,\cdots,\theta_k\}\subset [0,\frac{\pi}{2}]$ among two supspaces $S_1,S_2\in Gr_k(n)$ (see the paper by Li et al. \cite{qiu} for a nice explanation and applications). The subspaces are equal if and only if all the principal angles are zero. This allows one to define the \textit{angular distance} as $\dist(S_1,S_2)^2=\sum_{i=1}^k \theta_i^2$. It can be shown that this distance comes from a Riemanian metric (see for instance \cite[Theorem 4]{wong}). The Levi-Civita connection of this metric is well-known and makes of $Gr_k(n)$ a Riemannian symmetric space \cite{helga}. Its geodesics, paralell transport and curvature were calculated in various ways. One can also present $Gr_k(n)$ as the unitary orbit of a fixed $k$-dimensional orthogonal projection $P\in M_n(\mathbb K)$, for $\mathbb K=\mathbb R$ or $\mathbb K=\mathbb C$. In this case, the tangent space at $P$ is identified with the subspace of symmetric matrices $X$  (Hermitian in the complex case) such that $X=XP+PX$. Therefore a Riemannian metric is available, by means of $\langle X,Y\rangle_P=\textrm{Re}\mathrm{Tr}(XY^*)$. As early as 1948, Dixmier \cite{dixmier} found out that if $\|P-Q\|<1$ then there exists a unique $x^*=-x$ such that $X=[x,P]$ is tangent at $P$, with $\|X\|=\|x\|<\pi/2$ and $e^xPe^{-x}=Q$ (Dixmier called this matrix $x$ a \textit{direct rotation}). The eigenvalues of this direct rotation are the canonical angles among the subspaces identified with $\ran P,\ran Q$ (see \cite{qiu}). For the case of $\|P-Q\|=1$ he further showed that there exists a unitary $U$ such that $UPU^*=Q$ if and only if 
$$
\dim(\ker(P)\cap \ran(Q))=\dim(\ran(P)\cap \ker(Q)).
$$
The results of Dixmier and other consequences are better explained in a paper by Davis \cite{davis}. For a tangent vector $X$ one has $\langle X,X\rangle_P=\sum_{i=1}^k |\lambda_i(X)|^2$ with $\lambda_i(X)$ the eigenvalues of $X$, and this is the tangent  Riemannian metric that gives the angular distance in $Gr_k(n)$ as mentioned above. For its Levi-Civita connection, the unique geodesic through $P$ with initial speed $V=[v,P]$ (here $v^*=-v$ is $P$-codiagonal also) is $\gamma(t)=e^{tv}Pe^{-tv}=\Exp_P(tV)$.

Any reasonable norm with the $\theta_i$ would induce a Finsler metric, and then a distance in the Grassmannian by taking the infima of the length of rectifiable paths joining given endpoints. In terms of separation of subspaces in the Grassmannian, the one that is always available and controls every other metric is the supremum distance, given by
$$
\dist_{\infty}(P,Q)=\max\theta_i=\|x\|_{\infty}=\max\{\|x\xi\|: \|\xi\|=1\}
$$
when $Q=e^xPe^{-x}$ and $x$ is a direct rotation. Equivalently, one can give the tangent space the spectral norm of the whole space of matrices. With respect to the affine connection introduced above, this Finsler norm has the nice feature that makes of paralell transport an isometry. In particular geodesics have constant speed $\|V\|_{\infty}$ (the same remark applies to any unitarily invariant norm).

One of the first problems to solve is the charaterization of the conjugate points to $Q$ along $\gamma$, and the tangent conjugate locus of $P$, which are those $V$ such that $D(\Exp_P)_V$ has nontrivial kernel. This was solved for the classical Grassmannians $Gr_k(n)$ in a wonderful paper by Sakai in the seventies \cite{sakai}. It follows the ideas in the paper of Crittenden \cite{critten}, where the method of proof is based on the presentation of the Grassmanian as a symmetric space of the compact type, and using the machinery of Cartan subalgebras and real root decompositions. See also the paper by Berceanu \cite{berce} for further explanation and history of these developments.

\smallskip

In an abstract $C^*$-algebra $\mathcal A$, we let $Gr(P_0)\subset \mathcal A$ be a component of the Grasmann manifold presented as the unitary orbit of a self-adjoint projection $P_0$. These tools of root decompositions are not available, but one can presume that being an homogeneous manifold of the unitary group, the ideas and techniques of operator theory will provide a respectable substitute, following the approach by Andruchow, Corach, Porta, Recht and others \cite{andru1,cpr,pr,prV}. This is what we try to accomplish here, and we not only recover the tangent conjugate locus at $P\in Gr(P_0)$ but for each tangent conjugate point $TV\in T_PGr(P_0)$ we also give a full description of the kernel of the differential of the exponential map $(D\Exp_P)_{TV}$. 

\smallskip

In what follows we describe the organization ot this paper and we exhibit the main results along the way. First in Section \ref{cg} we present the main definitions and recall some results. Then for $\gamma$  a path in $Gr(P_0)$ and $\mu$ a vector field in $Gr(P_0)$ along $\gamma$, we can consider in $Gr(P_0)$ these four affine connections:
\begin{enumerate}
\item Let $\Pi$ the projection onto $T_PGr(P_0)$ along diagonal skew-adkoint operators: then $D_t\mu=\Pi (\mu'(t))$ is the \textit{horizontal connection} 
\item Let $U_t\subset \mathcal A$ be a horizontal lift of $\gamma$, i.e. $U_t'=[\gamma_t',\gamma_t]U_t$. Then $D_t\mu=U_t (U_t^*\mu U_t)'U_t$ is the \textit{reductive connection}
\item For $P\in Gr(P_0)$ let $S_P:Gr(P_0)\to Gr(P_0)$ be $S_P(Q)=(2P-1)Q(2P-1)$, then this makes of $Gr(P_0)$ a symmetric space and there is a natural torsion free connection that can be derived, this is the \textit{symmetric space connection}
\item If the algebra $\mathcal A$ has a finite trace, let $\langle X,Y\rangle=\mathrm{Tr}(XY^*)$, which makes of $Gr(P_0)$ a weak Riemannian manifold. There exist a Levi-Civita connection for this metric, the \textit{metric connection}.
\end{enumerate}
For the classical Grassmannians $Gr_k(n)$ it is well-known that $(1)$ (considered by Kovarik in \cite{kov} in the context of Banach algebras) and $(4)$ are the same connection, and furthermore it is the same connection of $(3)$ albeit with a different presentation of symmetric spaces (see for instance \cite{critten}). We show that the reductive connection, introduced by Porta and Recht in \cite{prV} is also the same connection hence \textit{all four connections above are in fact different presentations of the same one} (this is done in Proposition \ref{equalcon}, Section \ref{LC} and Remark \ref{ssc}). In Section \ref{sjf} we obtain a closed formula for any Jacobi field along a geodesic, for this connection. In Section \ref{scl} we show first that if $\mathcal A$ has real rank zero, normal geodesics are not minimizing past $|t|=\pi/2$, and then we discuss uniqueness of geodesics. This allows us to prove that normal geodesics are not minimizing past $t=\pi/2$ if $P_0$ has finite rank or co-rank (Corallary \ref{corocl}). Then we move on to the main subject of this paper, where we study the differential of $\Exp_P$ to find the conjugate points to $P$ along $\gamma$.  In Section \ref{scpoints} we prove
\begin{theorem*}[A]
Let $P\in Gr(P_0)$ and let  $V\in T_PGr(P_0)$ of unit speed. If $Q$ is conjugate to $P$ along $\gamma$ then $Q=\gamma(T)$ with 
$$
T=T(k,s,s')=\frac{k\pi}{|s-s'|}\qquad k\in\mathbb Z^*,\qquad s\ne s'\in \sigma(V).
$$
\end{theorem*}
This is shown in Lemma \ref{fcp} and the remarks before it. Now let $V=U|V|$ be the polar decomposition of $V$ in the enveloping von Neumann algebra of $\mathcal A$. Let 
$$
\lambda=(1-P)VP,\quad  \Omega=(1-P)UP.
$$
Let $P_{|\lambda|}$ stand for the projection onto the closure of the range of $\lambda$ and consider the $C^*$-algebra $\mathcal A_0=P_{|\lambda|}\mathcal A P_{|\lambda|}$. For the first conjugate point we obtain:
\begin{theorem*}[B] Let $V$ be a unit length tangent vector at $P\in Gr(P_0)$. Then the kernel of $D(\Exp_P)_{\frac{\pi}{2}V}$ at the first tangent conjugate point $Q=\gamma(\frac{\pi}{2})$ is
$$
\mathcal S=\left\{ \Omega z -z\Omega^*:\,z^*=-z\in \mathcal A_0\textrm{ and }|\lambda|z=z\right\}.
$$
If $Q$ is not monoconjugate to $P$, then it is epiconjugate to $P$.

If $d$ is the real dimension of the fixed point set of $V^2$ in any faithful representation of $\mathcal A$,  then for the complex Grassmannian we have that the dimension of $\mathcal S$ is $d^2$, while for the real Grassmanian it is $\frac{d^2-d}{2}$.
\end{theorem*}
Theorem (B) is contained in Theorem \ref{t1cp} and Corollary \ref{dimensiond}. We show an example where this $Q$ is not monoconjugate but epiconjugate (Example \ref{noesmono}), and another one where it is both (Example \ref{eslasdos}). For other candidates $Q=\gamma(T)$, the situation is different, since they might not be conjugated to $P$. However in Theorem \ref{mononcon} we obtain:
\begin{theorem*}[C] Let $T=T(k,s,s')$, let $\mu_j=\frac{j}{|k|}|s-s'|$ and 
$$
\Lambda=\{j\in\mathbb N: \exists s_1\ne s_2\in\sigma(V)\textrm{ with } j |s-s'|=|k||s_1-s_2|\},
$$
let 
$$
\mathfrak H=\oplus_{j\in\Lambda}\ker((L-R)^2-\mu_j^2)\big|_{(\mathcal A_0)_h},\qquad \mathfrak K=\oplus_{j\in\Lambda}\ker(L+R-\mu_j)\big|_{(\mathcal A_0)_{sk}}.
$$
Then $\ker(DExp_P)_{TV}=\mathcal S\oplus\mathcal T$, where
$$
\mathcal S=\left\{ \Omega (a+b) +(a-b)\Omega^*    : \,a\in\mathfrak H,\, b\in \mathfrak K\right\}
$$
and 
$$
\mathcal T=\{x=P_vx+xP_v\in T_PGr(P_0): \sinhc(TV)x=xP_v\}.
$$
\end{theorem*}
In Lemma \ref{muu2} we give sharp criteria to discard most of the direct summands in $\mathfrak H,\mathfrak K$. With these tools, in Section \ref{spc} we give a full characterization of the kernel at all conjugate points in the real and complex projective spaces (the orbit of a one-dimensional projection); this is Theorem \ref{losproye}. In Section \ref{sbc}, we study the points past the first tantent conjugate point. In the classical Grassmannians $Gr_k(n)$,  along a unit speed geodesic the point $Q=\gamma(T)$ for $T=(k,s,s')$ as above is always conjugate to $P$: more in general, we show in Lemma \ref{prime} that if $\mathcal A_0$ is a prime $C^*$-algebra or a von Neumann factor, then each $Q=\gamma(T)$ is either monogonjugate or epiconjugate to $P$. We end the paper with an elementary finite dimensional example that shows that when the center of $\mathcal A$ is not trivial, then most candidates to conjugate points (except for the first one) are in fact not conjugate to $P$ (Example \ref{pocos}).

\smallskip

\subsection*{Acknowledgements} This research was supported by Consejo Nacional de Investigaciones Cient\'\i ficas y T\'ecnicas (CONICET), Agencia Nacional de Promoci\'on de Ciencia y Tecnolog\'\i a (ANPCyT), and Universidad de Buenos Aires (UBA), Argentina.

\section{Connections, geodesics and Jacobi fields in the Grassmannian}\label{cg}

We will denote $\U_{\mathcal A}\subset \A$ the unitary group of the $C^*$-algebra $\A$. We shall denote by $\A_h\subset\A$ the real subspace  of self-adjoint (or Hermitian) elements of $\A$. Then $\A_{sk}=\mathrm{Lie}(\U)$ denotes the Lie algebra of the unitary group of $\A$, which consists of the skew-adjoint (i.e. anti-Hermitian) elements of $\A$. When $\mathcal A=\mathcal B(\mathcal H)$ is the algebra of bounded operators on a Hilbert space, one can identify a subspace $\mathcal S\subset \mathcal H$ with the unique orthogonal projection $P_{\mathcal S}\in \mathcal A$ such that $\ran(P_{\mathcal S})=\mathcal S$. After a motion by means of the unitary group $\mathcal S'=U(\mathcal S)$, the projections are transformed accordingly to $P_{\mathcal S'}=UP_{\mathcal S}U^*$. Thus we will consider the components of the Grassmann manifold, presented as unitary orbits of a fixed projection:
\begin{defi}The Grassmannian of the $C^*$-algebra $\A$ is the set
$$
Gr(\A)=\{P\in\A: P^2=P^*=P\}.
$$
\end{defi}
One advantage of this viewpoint of the Grassmannian is that it enables one to regard  $Gr(\A)$ as a subset of the  Banach space $\A$. Most of the facts concerning this viewpoint were presented in \cite{cpr,pr,prV}.  One sees the benefit of regarding subspaces as projections, when proving that this action just defined is locally transitive; this fact is well known (for instance \cite{kov,kato,halmos,davis,cpr,ass} etc.):
\begin{lem}\label{section}
Let $P,Q\in Gr(\h)$ such that $\|P-Q\|<1$. Then there exists $U=U(P,Q)\in\U_{\A}$ depending smoothy on $P,Q$ such that  $UPU^*=Q$. \end{lem}

Note the fact that for any pair $P,Q$ of projections, one always has $\|P-Q\|\le 1$. It follows that if we fix $P$, the set of $Q$ which are conjugate with $P$ contains an open dense subset of $Gr(\A)$, namely, $\{Q\in Gr(\A): \|P-Q\|<1\}$. 

\begin{defi} We will denote by $Gr(P_0)\subset Gr(\A)$ the orbit for the action of the unitary group of $\A$, of a fixed orthogonal projection $P_0^2=P_0^*=P_0\in\A$. We will further assume throughout the paper that $P_0$ is not a central projection, to avoid the trivial space.
\end{defi}

The following result is well-known, see for instance \cite{prV}:
\begin{teo}
$Gr(P_0)$ is a $C^\infty$ differentiable complemented submanifold of $\A_h$. For any fixed $P\in Gr(P_0)$, the orbit map $\pi_{P}:\U_{\mathcal A}\to Gr(P_0)$ given by $\pi_{P}(U)=UPU^*$ is a $C^\infty$ submersion.
 \end{teo}

\begin{rem}[Real Grassmannians] The computations of this paper can be carried out to the case of the real Grassmannian of a real Hilbert space $\mathcal H$, by identifying a subspace $\mathcal S\subset\mathcal H$ with the orthogonal projection onto $\mathcal S$. In this case, the symbol $x^*$ will be used to denote the transpose operator (for the finite dimensional case, it is just the transpose matrix). Thus the unitary group is replaced by the orthogonal group, and $Gr(P_0)$ is the similarity orbit for the action of the orthgonal group of $\mathcal H$. We can then include in our discussion real $C^*$-algebras (see \cite{rosenberg} for a systematic review of basic results in this setting). There are no significant modifications in the proofs, except for the fact that in some cases it is necessary to pass to the complexification to do some computations and then get back. Thus all the results hold without modifications (when necessary, we introduce some clarification). 
\end{rem}

\subsection{$P$ co-diagonal operators}\label{pcodiago}

Following ideas in \cite{pr} and \cite{cpr}, we shall base our study of the geometry of $Gr(P_0)$ on the following decomposition of $\A$. Fix $P\in Gr(P_0)$. Then operators $A\in \A$ can be written as $2\times 2$ block matrices:
$$
A=\left(\begin{array}{cc} PAP & PAP^\perp \\ P^\perp AP & P^\perp AP^\perp \end{array} \right)=\left(\begin{array}{cc} a_{11} & a_{12} \\ a_{21} & a_{22} \end{array} \right).
$$
Then $\A$ can be decomposed as
$$
A=\left(\begin{array}{cc} a_{11} & 0 \\ 0 & a_{22} \end{array} \right)+\left(\begin{array}{cc} 0 & a_{12} \\ a_{21} & 0 \end{array} \right)=A_d+A_c,
$$
where $A_d$ will be called the $P$-diagonal part of $A$, and $A_c$ the $P$-co-diagonal part of $A$. Note that $A_d$ commutes with $P$, and this characterizes the $P$-diagonal operators. Let us denote by $\D_{P}$ and  $\C_{P}$ the (closed, complemented) subpaces of $P$-diagonal and $P$-co-diagonal operators; clearly $\D_{P}\oplus\C_{P}=\A$. 
$$
[ {\D_P}, {\D_P}]\subset  {\D_P},\qquad [ {\D_P}, {\C_P}]\subset  {\C_P},\qquad [ {\C_P}, {\C_P}]\subset  {\D_P}.
$$

\begin{rem}[Symmetries] It shall be sometimes useful to consider the symmetry induced by a projection. Namely, projections are in one to one correspondence with symmetries $\mathfrak{s}\in\A$: $\mathfrak{s}^*=\mathfrak{s}^{-1}=\mathfrak{s}$. The correspondence is given by 
$$
P\longleftrightarrow \mathfrak{s}_P=2P-1.
$$
This viewpoint was also adopted by Porta and Recht \cite{pr} and also by Kovarik \cite{kov}.
\end{rem}

Returning to the diagonal/co-diagonal decomposition, note that the $P$-diagonal part $A_d$ commutes with $\mathfrak{s}_{P}$ and the $P$-co-diagonal part $A_c$ anti-commutes with $\mathfrak{s}_{P}$:
\begin{align}\label{anticonm}
A_c\mathfrak{s}_{_{P}}&=\left(\begin{array}{cc} 0 & a_{12} \\ a_{21} & 0 \end{array}\right)\left(\begin{array}{cc} 1 & 0 \\ 0 & -1 \end{array}\right)=\left(\begin{array}{cc} 0 & a_{12} \\ -a_{21} & 0 \end{array}\right)\\
&=-\left(\begin{array}{cc} 0 & a_{12} \\ a_{21} & 0 \end{array}\right)\left(\begin{array}{cc} 1 & 0 \\ 0 & -1 \end{array}\right)=-\mathfrak{s}_{P}A_c.\nonumber
\end{align}
Moreover, using this and a direct computations it is easy to obtain that
\begin{equation}\label{anticonmexp}
\mathfrak{s}_{_{P}}e^{-A_c}=e^{A_c}\mathfrak{s}_{_{P}}.
\end{equation}

Let us show how the diagonal / co-diagonal decomposition gives also a natural linear connection for $Gr(P_0)$. We shall make it here explicit, though it is a particular case of what in classical differential geometry is a reductive structure for a homogeneous space. 

\medskip

As discussed for Hermitian operators, the  anti-Hermitian operators $X^*=-X$ can also be decomposed as \textit{horizontal vectors} (co-diagonal with respect to $P$) and \textit{vertical vectors} (diagonal with respect to $P$). This induces a decomposition of $T\U$ (which is a vector bundle over $Gr(P_0)$) as the direct Whitney sum of a vertical bundle and a horizontal bundle. Let us develop these ideas with more detail:

\begin{lem}[Tangent spaces and vector fields in $Gr(P_0)$]\label{campos} Let $P\in Gr(P_0)$. Then
\noindent
\begin{enumerate}
\item
$X\in\C_P$ if and only if $UXU^*\in\C_{UPU^*}$, for any $U\in\U_{\mathcal A}$.
\item
$X\in\C_P$ if and only if $X=XP+PX$ if and only if $\si_P X=-X\si_P$; in this case $[P,[P,X]]=X$. 
\item If $X\in \C_P$, the spectrum of $X$ is balanced i.e. $\sigma(X)=- \sigma(X)$.
\item 
The tangent space $T_P Gr(P_0)$ of  $Gr(P_0)$ is $\C_{P}\cap \A_h$, and a typical tangent vector $X_P$ at $P$ is of the form $X_P=[x,P]$ with $x^*=-x\in \C_P$; such $x$ is unique and is in fact given by $x=[X_P,P]$. 
\item In particular $x\mapsto [x,P]$ is an isomorphism between  $\C_P\cap \A_{sk}$ and $T_PGr(P_0)$, with inverse $V\mapsto [V,P]$.
\item Let $X$ be a vector field in $Gr(P_0)$, which is a smooth map from $Gr(P_0)$ into $\A_h$ such that $X_P\in \C_P$. Then it verifies 
$$
-DX_P(Y_P)=P(X_PY_P+Y_PX_P)=(X_PY_P+Y_PX_P)P
$$
and
$$
DX_P(Y_P)=DX_P(Y_P)P+PDX_P(Y_P)+X_PY_P+Y_PX_P
$$
for any $Y_P\in \C_P$.
\end{enumerate}
\end{lem}
\begin{proof}
The first two assertions can be verified by hand from the very definitions and equation (\ref{anticonm}). Assumme that $\lambda\in \sigma(X)$, then $X-\lambda$ is not invertible and the same holds for $\si_P(X-\lambda)$. But
$$
\si_P(X-\lambda)=-X\si_P-\lambda\si_P=-(X+\lambda)\si_P
$$
thefore $X+\lambda$ is not invertible. The fourth assertion is a consequence of: the first two assertions and the fact that $\pi_P$ is a smooth submersion with kernel the diagonal operators $A^*=-A$. The fifth assertion is a restatement of the previous. Now let $X$ be a vector field in $Gr(P_0)$, consider any smooth path $\alpha$ in $Gr(P_0)$ with $\alpha(0)=P$ and $\alpha'(0)=Y_P$. Write $X'=DX_P(Y_P)=(X\circ\alpha(t))'|_{t=0}$, for short, also write $X=X_P$ and $Y=Y_P$. Since $PX_PP=0$ for any $P\in G(\h)$, replacing $P$ with $\alpha$ and differentiating at $t=0$ we get $PX'P+YXP+PXY=0$. Now since $\mathfrak s_P YX=-Y\mathfrak s_P X=YX\mathfrak s_P$, we also have $PYX=YXP$, thus $-PX'P=P(YX+XY)=(XY+YX)P$ and this proves the first identity in (6). The second identity has a similar proof, now starting from $X_P=PX_P+X_PP$.
\end{proof}

\begin{rem}[Reductive subalgebra]\label{csub} Using the tildes to denote the diagonal and co-diagonal skew-adjoint operators, we note that we have a decomposition 
$$
Lie(\U_{\A})=\A_{sk}=\widetilde{\D_P}\oplus \widetilde{\C_P}
$$
which is graded, in the sense that
$$
[\widetilde{\D_P},\widetilde{\D_P}]\subset \widetilde{\D_P},\qquad [\widetilde{\D_P},\widetilde{\C_P}]\subset \widetilde{\C_P},\qquad [\widetilde{\C_P},\widetilde{\C_P}]\subset \widetilde{\D_P}.
$$
This structure induced by the action (in particular, the last inclusion) is also what is usually called \textit{reductive}.
\end{rem}

\subsection{Linear connections} In this section we discuss the connection of the Grassmannian, which comes in different presentations.

\begin{defi}[The horizontal connection]\label{hc}
Let ${\bf P}_{\C_P}:\A_h\to\C_P\subset\A_h$ be the projection onto $\C_P$ along diagonal operators, it easy to see that it is given by
$$
{\bf P}_{\C_P}(A)=PAP^\perp+P^\perp AP \hbox{ for } A\in\A_h.
$$
Thus for any vector field $X$ in $Gr(P_0)$ we can define the \textit{horizontal affine connection} by differentiating and projecting, that is if $X$ is a vector field along a smooth path in $Gr(P_0)$, we let
$$
\frac{D^H X}{dt} ={\bf P}_{\C_{P}}(\dot{X}(t)).
$$
This connection was introduced in this form by Kovarik \cite{kov}.
\end{defi}

We now discuss a different presentation of this connection, which can be derived from the general presentation of the Grassmannian as an homogeneous reductive space as shown by Mata-Lorenzo, Porta and Recht in \cite{rm,prV}.

\begin{defi}
Let $P(t)$ with $t\in I=[t_0,t_1]$ be a smooth curve in $Gr(P_0)$. A smooth curve $U(t), t\in I$ of unitary operators is a \textit{co-diagonal lifting}  for $P(t)$, if 
$$
U(t)P(t_0)U(t)^*=P(t) \ \ \hbox{ and }\ \ U^*(t)\dot{U}(t)\in\C_{P(t)}.
$$
\end{defi}
If one requires that $U(t_0)=1$, then the curve $U$  is unique.  Existence and uniqueness of such liftings follow from the next result, for a proof see \cite{prV}:
\begin{lem}
The co-diagonal lifting satisfying $U(t_0)=1$, is  the unique solution of the following linear differential equation:
\begin{equation}\label{levantadahorizontal}
\left\{ \begin{array}{l} \dot{U}=[\dot{P}_t,P_t] U \\ U(t_0)=1 \end{array} \right. ,
\end{equation}
where we abbreviate $P_t=P(t)$.
\end{lem}
 
With the aid of the co-diagonal lifting we define a notion of parallel transport of tangent vectors in $Gr(P_0)$:
\begin{defi}[Paralell transport]
Given $X\in T_P Gr(P_0)$ and $P:[0,1]\to Gr(P_0)$,a smooth curve in $Gr(P_0)$ with $P(0)=P$, the parallel transport of $X$ along $P(t)$ is defined as $X\mapsto U(t)XU^*(t)$, where $U$ is the solution of the equation (\ref{levantadahorizontal}), $\dot{U}_t=[\dot{P}_t, P_t]U_t$, $U(0)=1$. 
\end{defi}

\begin{rem}\label{pt} It is easy to check that if $\delta(t)=e^{tZ}Pe^{-tZ}$ with $Z^*=-Z$ a co-diagonal operator, then the co-diagonal lifting of $\delta$ is $U(t)=e^{tZ}$. With this, paralell transport of $X_P\in \C_P=T_PGr(P_0)$ along $\delta$ is given by 
$$
\mu(t)= e^{tZ}X_P e^{-tZ}.
$$
\end{rem}

Note that by the Lemma \ref{campos} above, $X\in T_P Gr(P_0)=\C_{P}$ implies that 
$UXU^*\in \C_{UPU^*}=T_{P_t} Gr(P_0)$. The notion of parallel transport just introduced induces a covariant derivative as follows: 
\begin{defi}[The reductive connection]
If  $X(t)$ is a smooth field of tangent vectors along $P(t)$ in $Gr(P_0)$, let
\begin{equation}\label{derivadacovariante}
\frac{D^R X}{d t}=U\{\frac{d}{d t}(U^*XU)_{t=0}\}U^*,
\end{equation}
where $U$ is the solution of equation (\ref{levantadahorizontal}): $\dot{U}=[\dot{P}, P]U$, $U(0)=1$. Note that $U^*XU$ is a smooth curve in $\C_{P_0}$, thus its derivative is an element in $\C_{P_0}$, and therefore $\frac{D X}{d t}\in\C_{P(t)}=T_{P_t} Gr(P_0)$. We shall call this connection the \textit{reductive connection}, see \cite{rm}. 
\end{defi}

\begin{prop}[The reductive connection equals the horizontal connection]\label{equalcon} For $X,Y$  smooth vector fields in $Gr(P_0)$, $P\in Gr(P_0)$, the connection can be computed as
\begin{align*}
\nabla_Y X(P) & =DX_P(Y_P)+ \mathfrak s_P(X_PY_P+Y_PX_P)\\
 &= DX_P(Y_P)+[X_P,[Y_P,P]]=DX_P(Y_P)+[Y_P,[X_P,P]].
\end{align*}
Therefore it is a connection without torsion whose Christoffel bilinear operator at $P$ is given by $\Gamma_P(X,Y)=[X,[Y,P]]=\mathfrak s_P(XY+YX)$ for $X,Y\in \C_P\cap \A_h$.
\end{prop}
\begin{proof}
Denoting $X'=DX_P(Y_P)$, by Lemma \ref{campos}(4), if we compute ${\bf P}_{\C_P}(DX_P(Y_P))$ we obtain 
\begin{align*}
\frac{D^H X}{dt} & =-2PX'P+PX'+X'P=2P(X_PY_P+Y_PX_P)+X'-(X_PY_P+Y_PX_P)\\
& =X'+\mathfrak s_P(X_PY_P+Y_PX_P).
\end{align*}
This establishes the first identity. On the other hand, writing $X=X_P,Y=Y_P$ for short we have
\begin{align*}
[X,[Y,P]] &=(XY+YX)P-XPY-YPX\\
& =P(XY+YX)-1/2(X\mathfrak s_PY+ XY)-1/2(Y\mathfrak s_PX+ YX)\\
& =P(XY+YZ)+1/2 \mathfrak s_P (XY+YX)-1/2(XY+YX)\\
&=(2P-1)(XY+YX)=\mathfrak s_P(XY+YX)
\end{align*}
where we used several times that $\mathfrak s_PX=-X\mathfrak s_P$ which implies $PXY=XYP$ (Lemma \ref{campos}(2)). This proves the second equality, and exchanging $X,Y$ in the previous equation we also obtain the proof of the final equality of the statement. Thus the formulas or this proposition give the horizontal connection. Now we show that it also equals the reductive connection: the value of a covariant along a path at $P$ only depends on the speed of the path at $P$, so let $\delta(t)=e^{tZ}Pe^{-tZ}$ with $Z^*=-Z$ co-diagonal and $[Z,P]=Y_P\in \C_P$; by the previous remark its horizontal lift is $U(t)=e^{tZ}$. Let $X$ be any vector field and consider $X$ along $\delta$, i.e. $X(\delta(t))$. Then by defninition of the reductive connection, we have that its value at $P$ is
$$
\frac{D^R X}{dt}=-ZX_P+X'_0+X_PZ=DX_P([Z,P])+[X_P,Z]=DX_P(Y_P)+[X_P,[Y_P,P]]
$$
where the last equality is due to Lemma \ref{campos}(3).
\end{proof}

\begin{coro}The geodesics of $Gr(P_0)$ for the connection are $\delta(t)=e^{tv}Pe^{-tv}$ with $v^*=-v$ co-diagonal, and paralell transport of $X\in \C_P$ along $\delta$ is given by $\mu(t)=e^{tZ}Xe^{-tZ}$. The exponential map of the connection  $Exp_P: \C_P\to Gr(P_0)$ is given by $Exp_P(V)=e^{[V,P]}Pe^{-[V,P]}$.
\end{coro}

\begin{rem}\label{blm} For given $P\in Gr(P_0)$, we represent any operator as a $2$ by $2$ block matrix as in Section \ref{pcodiago}, i.e. 
\begin{equation}\label{VX}
P=\left(\begin{array}{cc}1 & 0 \\ 0 & 0 \end{array}\right),\quad  V=\left(\begin{array}{cc} 0 & \lambda  \\ \lambda^* & 0 \end{array}\right), \quad X=\left(\begin{array}{cc} 0 & \rchi  \\ \rchi^* & 0 \end{array}\right)
\end{equation}
with $X,V$ typical tangent vectors at $P$, with $\lambda,\rchi$ rectangular block operators. Moreover if $v=[V,P]$ and $x=[X,P]$ then
$$
v=\left(\begin{array}{cc} 0 & -\lambda  \\ \lambda^* & 0 \end{array}\right), \quad x=\left(\begin{array}{cc} 0 & -\rchi  \\ \rchi^* & 0 \end{array}\right)
$$
and $[v,P]=V$, $[x,P]=X$. Since
$$
v^*v=\left(\begin{array}{cc}  \lambda\lambda^* & 0 \\ 0 & \lambda^*\lambda \end{array}\right)=V^*V
$$
it is apparent that $\|v\|=\|V\|$. Moreover $\|V\|=1$ if and only $\|\lambda^*\lambda\|=1 \,(=\|\lambda\lambda^*\|)$.
\end{rem}

Let $\sinc$ denote the cardinal sine function and $\sinhc$ denote the cardinal hyperbolic sine function, i.e.
$$
\sinc(x)=\left\{
\arraycolsep=4pt\def\arraystretch{1.1}\begin{array}{crl}
\displaystyle x^{-1}\sin(x) & \mathrm{if} & x\ne 0\\
1  & \mathrm{if} & x=0 \end{array}  \right.  \qquad
\sinhc(x)=\left\{
\arraycolsep=4pt\def\arraystretch{1.1}\begin{array}{crl}
\displaystyle x^{-1}\sinh(x) & \mathrm{if} & x\ne 0\\
1  & \mathrm{if} & x=0 \;.\end{array}  \right.
$$
\begin{rem}Let $V\in T_PGr(P_0)$ be as in (\ref{VX}), then the unique geodesic from $P$ with initial speed $V$ is
$$
\gamma(t)= \left(
\arraycolsep=-1pt\def\arraystretch{1.4}
\begin{array}{cc} \cos^2(t|\lambda^*|) & t\lambda \cos(t|\lambda|)\sinc(t|\lambda|)   \\ t\cos(t|\lambda|)\sinc(t|\lambda|)\lambda^* & \sin^2(t|\lambda|) \end{array}\right).
$$
\end{rem}
The proof is elementary but we include it for completeness: from the previous remark we know that
$$
v^2=-v^*v=\left(\begin{array}{cc}  -\lambda\lambda^* & 0 \\ 0 & -\lambda^*\lambda \end{array}\right), \quad\mathrm{ thus  }\quad v^{2k}=\left(\begin{array}{cc}  (-1)^k |\lambda^*|^{2k} & 0 \\ 0 & (-1)^k|\lambda|^{2k} \end{array}\right),
$$
and likewise
$$
v^{2k+1}=v^{2k}v=\left(\begin{array}{cc} 0& \!\!\! (-1)^{k+1} |\lambda^*|^{2k}\lambda \\  (-1)^{k}|\lambda|^{2k}\lambda^*  & 0\end{array}\right)=\left(\begin{array}{cc} 0&  \!\!\! (-1)^{k+1} \lambda |\lambda^*|^{2k} \\  (-1)^{k}\lambda^*|\lambda|^{2k}  & 0\end{array}\right).
$$
Using the exponential map series and with some rearranging, the formulas follow.

\subsection{Levi-Civita connection of the trace}\label{LC}

Let us  consider the case when the algebra has a finite trace $\tau$. In this case, there is an inner product defined globally on $\A$: 
$$
\langle X,Y\rangle=\textrm{Re}\;\tau(XY^*),
$$
and a norm $\|X\|_2=\sqrt{\tau X^*X}$. Note that $\tau(A^*)=\overline{\tau(A)}$, and thus $\langle X,Y\rangle\in\mathbb{R}$ if $X^*=X$ and $Y^*=Y$:
$$
\overline{\langle X,Y\rangle}=\overline{\tau(XY)}=\tau((XY)^*)=\tau(YX)=\tau(XY)=\langle X,Y\rangle.
$$
Therefore, a natural Riemannian metric is defined in $Gr(P_0)$, by endowing every tangent space with the fixed inner product of the trace in $\A_h$. When $\A$ is infinite dimensional, this is a \textit{weak} Riemannian metric: it is continuous but the trace norm does not give the original topology of the Grassmannian, which is the $C^*$-algebra topology.

\smallskip

Recall  ${\bf P}_{\C_P}:\A_h\to\C_P\subset\A_h$ is the projection along diagonal operators. We claim that diagonal operators are orthogonal to co-diagonal operators, or equivalently, that ${\bf P}_{\C_P}$ is an orthogonal projection:  
\begin{align*}
\langle{\bf P}_{\C_P}(A),B\rangle & =Tr((PAP^\perp+P^\perp AP)B) =Tr(PAP^\perp B)+Tr(P^\perp APB)\\
&=Tr(AP^\perp BP)+Tr(APBP^\perp) =Tr(A(P^\perp BP +PBP^\perp))\\
& =\langle A,{\bf P}_{\C_P}(B)\rangle.
\end{align*}
Thus the horizontal/reductive connection is the Levi-Civita connection of the trace Riemannian metric in $Gr(P_0)$ (see \cite{kov}). 

\medskip

It is well known and not hard to check that for $X,Y,Z\in \C_P$, the curvature tensor of the reductive connection is 
\begin{equation}\label{curvat}
R_P(X,Y)Z=-[[X.Y],Z].
\end{equation} 
Thus in the presence of a trace, if  $X,Y\in \C_P$ are orthonormal for the trace inner product, then the sectional curvature of the $2$-plane generated by $X,Y$ is 
$$
1/2\, Sec_P(X,Y)=-\langle XY,YX\rangle+\|XY\|_2,
$$
which is non-negative because of the Cauchy-Shwarz inequality of the trace, and in fact it is zero if and only if $[X,Y]=0$.

\subsection{Complex structure and symplectic form}

The Grassmannian has a natural complex structure $J:TGr(P_0)\to TGr(P_0)$ given by  
$J([X,P])=-iX$ for $X^*=-X$ co-diagonal. Equivalently, if $X_P\in C_P$ the complex structure is computed as
$$
J(X_P)=i(PX_P P^{\perp}-P^{\perp}X_PP)=i[P,X_P].
$$
This complex structure is equivariant in the sense that $UJ(X)U^*=J(UXU^*)$ and skew-hermitian with respect to the trace inner product (in the case of an algebra with trace). In this case the orbit also has a canonical symplectic form (the KKS-form, for Kostant, Kyriliov and Sorieau), given by
$$
\omega_P(X_P,Y_P)=-\langle X_P, JY_P\rangle=-\tau(X_P \, JY_P)=\tau(P[X,Y])
$$
for $X_P=[X,P],Y_P=[Y,P]\in \C_P$. This $2$-form is closed, if one uses the intrinsic definition for the exterior differential of a $2$-form given by 
$$
d\omega(X,Y,Z)=X(\omega(Y,Z))-\omega([X,Y],Z)+ \textrm{ the other two cyclic permutations}.
$$
Thus $(Gr(P_0), \,J\, , \langle\; ,\;\rangle \,,\, \omega)$ is a weak K\"ahler manifold (see \cite{tumpach} and the references therein). It is in fact strongly  K\"ahler or even hyper-K\"ahler in some cases (see \cite{tumpach} and the references therein), in particular for finite dimensional algebras $\mathcal A$.

Since the metric, the complex structure and thus the symplectic form are invariant for the coadjoint action of the unitary group $\U$, one sees that the group acts on the orbit by symplectomorphisms. The action admits the equivariant moment map $\mu:Gr(P_0)\to \A_{sk}^*$, which is essentially the identity. i.e. let us denote $\mu(P)(X)=\mu^X(P)$, then
$$
\mu^X(P)=\langle P,X\rangle=\tau(PX)
$$
for $P=UP_0U^*\in Gr(P_0)$ and $X^*=-X\in Lie(\U)=\mathcal A_{sk}$.  If we denote $X_{Gr}$ the vector field induced by the action, i.e. $X_{Gr}(P)=[X,P]$, we can see that the action is Hamiltonian: by derivating $\mu^X(U_tPU^*_t)=\tau(U_tPU_t^*X)$ with $U_t=e^{tZ}Pe^{-tZ}$ we obtain
$$
(D\mu^X)_P([Z,P])=\tau([Z,P]X)=-\tau(P[X,Z])=-\omega(X_{Gr}, [Z,P]).
$$
Therefore the map  $\mu^X:Gr(P_0)\to\mathbb R$ is a Hamiltonian: $D\mu^X=-i_{X_{Gr}}\omega$.

\subsection{Symmetric space structure} Consider $\sigma:\mathcal U_{\mathcal A}\to \mathcal U_{\mathcal A}$ given by $\sigma(u)=\mathfrak s u\mathfrak s$ with $\mathfrak s=2P_0-1$ the associated involution. Then $\sigma$ is an automorfism and it is involutive ($\sigma^2=id$). The fixed point subgroup of $\sigma$ is 
$$
\U_D=\{u\in\U_{\mathcal A}: uP_0=P_0u\},
$$ 
the diagonal unitary elements of $\mathcal A$. Then since $Gr(P_0)\simeq \U_{\mathcal A}/\U_{\mathcal D}$, the Grassmannian is a symmetric space in the sense of Cartan-Loos, see \cite[Section 3]{neeb} and \cite[Section 1.3]{chu}. The induced symmetric space action is given by the formula
$$
\mu(P,P')=\mathfrak s_P P' \mathfrak s_P
$$
for  $P,P'\in Gr(P_0)$ where $\mathfrak s_P=2P-1$. If we denote $S_P(P')=P\cdot P'=\mu(P,P')$, this is a Banach symmetric product satisfying 
\begin{enumerate}
\item $P\cdot P=P$ and $\forall P\in Gr(P_0)$ there exists a neighbourhood $V$ such that if $P'\in V$ and $P\cdot P'=P'$, then $P=P'$.
\item $P\cdot (P\cdot P')=P'$
\item $P\cdot (P'\cdot P'')=(P\cdot P')\cdot (P\cdot P)''$
\end{enumerate}
One obtains a symmetric product in the tangent bundle by letting $\Sigma=T\mu$ (which makes of $TGr(P_0)$ also a Banach symmetric space). The notation $\Sigma_V=\Sigma(V,\cdot)$ is used for the symmetries of $TGr(P_0)$. In our case by differentiating $\mu$ we get 
$$
\Sigma_V(W)= 2V P' \mathfrak s_P  + 2\mathfrak s_P P' V+ \mathfrak s_PW\mathfrak s_P.
$$
for tangent vectors $V,W$ at $P,P'$ respectively.

\begin{rem}[Symetric space connection]\label{ssc} Each symmetric product carries a natural torsion free connection, where its spray is given by $F(V)=-D(\Sigma_V\circ Z)(V)$, being $Z:T\to TM$ the zero section of the manifold (see \cite[Section 3]{neeb} for the Banach manifold setting). In our case $\Sigma_{V/2}\circ Z(P)=V P' \mathfrak s_P  + \mathfrak s_P P' V$, and if we differentiate with respect to $P'$ at $P$ in the direction of $V$ we obtain $F(V)=-2V^2\mathfrak s_P=-\mathfrak s_P 2V^2$ (the last equality by Lemma \ref{campos}(2)). Now the bilinear Christoffel  operator is obtained by polarizing the spray $F$, i.e.
$$
\Gamma_P(V,W)=-\frac{1}{2}(VW+WV)
$$
for tangent vectors $V,W$ at $P\in Gr(P_0)$ (see \cite[Chapter VIII]{lang}). We can use the standard formula to recover the covariant derivative from the spray, which for a vector field $\mu$ along the path $\gamma\subset M$ is $D_t\mu= \mu'-\Gamma_{\gamma}(\gamma',\mu)$. In this case we obtain 
$$
D_t\mu=\mu'+\mathfrak s_{\gamma}(\gamma'\mu +\mu\gamma')
$$
thus by Proposition \ref{equalcon} \textit{the connection induced by the symmetric product in the Grassmannian equals the reductive connection}.
\end{rem}

All the geometric properties of the connection can be derived from the symmetric product $\mu$. In particular, the map $S_P$ is the \textit{geodesic symmetry} around $P\in Gr(P_0)$: it maps $\gamma(t)$ to $\gamma(-t)$ for any geodesic $\gamma$ of the connection such that $\gamma(0)=P$. In fact, the reductive connection is the unique torsion free connection in $Gr(P_0)$ which is invariant for all the symmetries $S_P$, $P\in Gr(P_0)$ \cite[Theorem 3.6]{neeb}. As a consequence, $Gr(P_0)$ is locally symmetric in the classical sense: if $R$ is the curvature tensor of the connection, then $\nabla R=0$ \cite[Teorema 7.5.8]{lar}.

\subsection{Jacobi fields}\label{sjf}

In this section we discuss Jacobi fields in order to get a better understanding of the conjugate points along geodesics of the Grassmannian. We first recall the following facts (for proofs see for instance \cite[Theorem 2.9]{lang}):

\begin{prop}Let $\nabla$ be an affine connection in a Banach manifold $M$, let $R$ be its curvature tensor. Let $\gamma$ be a geodesic of $\nabla$ and let $\mu$ be a vector field along $\gamma$. We indicate with $D_t$ the covariant derivative of $\mu$ along $\gamma$. The following are equivalent:
\begin{enumerate}
\item $\mu$ is a Jacobi field along $\gamma$, i.e. $D^2_t \mu= R_{\gamma}(\gamma',\mu)\gamma$.
\item There exists a variation $\nu_s$ of $\gamma$ by geodesics such that $\mu=\nu'_0$.
\end{enumerate}
\end{prop}

We also need to recall the well-known formulas for the differential of the exponential map in a Banach-Lie group:
\begin{lem}Let $v,w\in Lie(G)$ where $G$ is a Banach-Lie group, let $\exp$ be its exponential map, let $F(\lambda)=\frac{1-e^{-\lambda}}{\lambda}$ extended by $1$ at $\lambda=0$, let $G(\lambda)=e^{\lambda}F(\lambda)$. Then
$$
D\exp_v(w)=e^v F(\ad v)w=[G(\ad v)w]e^w.
$$
\end{lem}

In the following proposition we assumme that $v,x,y$ are skew-adjoint and co-diagonal with respect to $P\in Gr(P_0)$.

\begin{teo}[Jacobi fields]\label{JF} Let $\gamma(t)=e^{tv}Pe^{-tv}$ be the unique geodesic of the reductive connection of $Gr(P_0)$ with $\gamma(0)=P$, $\gamma'(0)=[v,P]=V$. Let $X=[x,P], Y=[y,P]$. Then the unique Jacobi field $\mu$ along $\gamma$ with $\mu(0)=X$, $D_t\mu(0)=Y$ is given by
\begin{align*}
\mu(t)&= [x, e^{tv}Pe^{-tv}]+ e^{tv}[t F(t \ad v)y,P]e^{-tv}\\
& = e^{tv} \left\{ [\cosh(t  \ad v)x,P] +t[\sinhc(t\ad v)y,P]  \right\}e^{-tv}.
\end{align*}
\end{teo}
\begin{proof}
Consider $\beta_s=e^{sx}Pe^{-sx}$, the geodesic with initial speed $X$, and let $\xi_s=P_0^s(\beta)(V+sY)\in T_{\beta_s}Gr(P_0)$, then $\xi_0=V$. Let $\nu_s(t)=\Exp_{\beta_s}(t\xi_s)$, which is clearly a variation of $\gamma$ by geodesics, so $\mu(t)=\frac{d}{ds}|_{s=0}\nu(s,t)$ is a Jacobi field along $\gamma$. Clearly also
$$
\mu(0)=\frac{d}{ds}\big|_{s=0}\nu(s,0)=\frac{d}{ds}\big|_{s=0}\Exp_{\beta_s}(0)=\frac{d}{ds}\big|_{s=0}\beta_s=\beta'(0)=X.
$$
Since $\xi_s=P_0^s(\beta)(V)+sP_0^s(\beta)(Y)$, we have that 
$$
\xi_0'=\Gamma_P(X,V)+Y+0=[X,[V,P]]+Y=[X,v]+Y.
$$
The first equality is due to the definition of paralell transport and the second is due to the formula for the connection given in Proposition \ref{equalcon}. The third equality is due to Lemma \ref{campos}(4). Now $\frac{d}{dt}|_{t=0}\nu(s,t)=\xi_s$ hence
\begin{align*}
D_t\mu(0) & =D_t|_{t=0}\nu_0'(t)  =\frac{d}{dt}\big|_{t=0}\frac{d}{ds}\big|_{s=0}\nu(s,t)- \Gamma_P(X,V)\\
&=\frac{d}{ds}\big|_{s=0}\xi_s -\Gamma_P(X,V)=\xi_0'-[X,v]=Y.
\end{align*}
By the uniqueness of solutions of differential equations in Banach spaces, $\mu$ is the unique Jacobi field along $\gamma$ with the prescribed condition. Now we compute $\mu$ explicitly: by Remark \ref{pt}, we have that  $\xi_s=[e^{sx}(v+sy)e^{-sx},\beta_s]$, hence 
$$
\nu_s(t)=e^{sx}e^{t(v+sy)}Pe^{-t(v+sy)}e^{-sx}.
$$
Now
\begin{align*}
\nu_0'(t) & =xe^{tv}Pe^{-tv}-e^{tv}Pe^{-tv}x+\exp_{*tv}(ty) Pe^{-tv}-e^{tv}P\exp_{*-tv}(-ty)\\
&  = [x, e^{tv}Pe^{-tv}]+ e^{tv}[t F(t \ad v)y,P]e^{-tv}\\
&=e^{tv}\left\{[e^{-t \ad v}x,P] + t [F(t \ad v)y,P]\right\}e^{-tv},
\end{align*}
where in the second equality (last term) we used the formulas for the differential of the exponential map in the Lie group of unitary operators. Now we have to recall that the Lie algebra of skew-adjoint operators has a grading, the Cartan decomposition of the algebra in diagonal and co-diagonal operators (with respect to $P$, see Remark \ref{csub}). The following identities
$$
e^{\lambda}=\cosh(\lambda)+\sinh(\lambda)\qquad  1-e^{-\lambda}=\sinh(\lambda)+ 1-\cosh(\lambda)
$$
gives us the diagonal-codiagonal decomposition of $e^{\ad v}$ and of $F(\ad v)$, and since the diagonal elements of $\A_{sk}\cap \mathcal D_P$ are in the kernel of $\ad P$, this finishes the proof.
\end{proof}

\begin{rem}[Killing fields]A vector field $X$ is Killing if its flow $\rho_t$ is an automorphism of the reductive connection. It is well-known that $X$ is Killing if and only if $X\circ \gamma$ is Jacobi along $\gamma$ for any geodesic $\gamma$ of the connection. Fix $P\in P_0$: if $\gamma(t)=e^{tv}Pe^{-tv}=e^{t\ad v}P$ as before, note that $X(Q)=[x,Q]$ (with $Q\in Gr(P_0)$) is a Killing field when $x^*=-x\in \C_P$, in fact it is the unique Killing field with $X(P)=[x,P]$. The flow of this Killing field is the one-parameter group of connection automorphisms
$$
\rho_t(Q)=e^{tx}Qe^{-tx}.
$$
\end{rem}

\section{Cut locus and conjugate locus}\label{scl}

In this section we discuss the cut locus and conjugate points along a geodesic, and compute the order of degeneracy of the first conjugate points. Then since the Hopf-Rinow theorem is not valid in infinite dimensions, we discuss existence and uniqueness of geodesics joining given endpoints.

\begin{defi}[Length and distance] We will measure the length of paths with the norm of the $C^*$-algebra, which is the spectral norm. Then the rectifiable distance $\dist_{\infty}$ is defined accordingly as the infima of the length of the paths joining given endpoints. This will be called the \textit{spectral distance} in $Gr(P_0)$ and denoted $\dist_{\infty}(P,Q)$ for $P,Q\in Gr(P_0)$. We say that a path is \textit{minimizing} if its length equals the distance among its endpoints.
\end{defi}

\begin{teo}[Minimizing geodesics]\label{min} If $V=[v,P]\in T_PGr(P_0)$ with $v\in \widetilde{\mathcal C}_P$ and $\|V\|=1$, then the unique geodesic $\gamma(t)=e^{tv}Pe^{-tv}$ from $P$ with initial speed $V$ is minimizing for $t\in [-\frac{\pi}{2},\frac{\pi}{2}]$. 
\end{teo}
\begin{proof}
This result for the spectral norm was proved by Porta and Recht in \cite{pr}.
\end{proof}

Note that in particular the geodesic diameter of the Grassmannian of a $C^*$-algebra is greater or equal than $\pi/2$ (recall that we assumed throughout that $P_0$ is non-central).

\subsection{Cut locus}

\begin{defi} The \textit{cut locus}  $P\in Gr(P_0)$ is the set of points $Q\in Gr(P_0)$ such that geodesics from $P$ to $Q$ are not minimizing past $P$. The \textit{tangent cut locus} $TC_P\subset T_PM$  is the preimage of $C_P$ at $P$ by means of the exponential map i.e. $TC_P=\{V\in T_PM: \Exp_P(V)\in C_P\}$. 
\end{defi}

Following an idea in \cite{acl} we show next that in some cases the cut locus is exactly at the first tangent conjugate point (see Theorem \ref{t1cp} below), as in the classical setting. We recall that $\mathcal A$ has \textit{real rank zero} when the set of self-adjoint elements with finite spectrum is dense in the space of self-adjoint elements of $\mathcal A$. The first examples of such $C^*$-algebras are the von Neumann algebras or the compact operators on a separable Hilbert space.

\begin{teo}\label{notmin} Asumme that $\mathcal A$ has real rank zero. Then unit speed geodesics of $Gr(P_0)$ are not minimizing past $|t|=\frac{\pi}{2}$.
\end{teo}
\begin{proof}
Let $\gamma(t)=e^{tv}Pe^{-tv}$ with $\|v\|=1$, assumme that $t_0>\pi/2$, let $v_0=t_0 v$, then $\|v_0\|>\pi/2$. Let  $Q=\gamma(t_0)=e^{v_0}Pe^{-v_0}$. Let $v_n^*=-v_n\in \mathcal A$ be such that $v_n$ has finite spectrum and $\|v_n-v_0\|<\frac{1}{n}$, let $Q_n=e^{v_n}Pe^{-v_n}$.  Consider the truncation $z_n$ of $v_n$ into the interval $i[-\frac{\pi}{2},-\frac{\pi}{2}]$, i.e. for $k\in \mathbb N_0$ let 
\[ f(x)=\left\{
\arraycolsep=-1pt\def\arraystretch{1.2}\begin{array}{lcc}
 x +(k+1)\frac{\pi}{2}  & \,\, \, \, \, \, \,  \, \, \, \, \, -2(k+2)\frac{\pi}{2} \leq x < -(k +1)\frac{\pi}{2} \\
 x & \, \, \, \, \, \,  \, \, \, \, \, -\frac{\pi}{2} \leq x \leq \frac{\pi}{2} \\
  x-(k+1)\frac{\pi}{2} & \, \, \, \, \, \, \, \, \, \quad\; (k+1)\frac{\pi}{2} < x \leq (k+2)\frac{\pi}{2}
\end{array}
\right.,
\]
and let $z_n=if(-i v_n)=-z_n^*$. Since the spectrum of $v_n$ is finite, $z_n\in \mathcal A$, and moreover $\|z_n\|\le \pi/2$. It is plain also that $e^{z_n}=e^{v_n}$, thus $Q_n=e^{z_n}Pe^{-z_n}$ and if we let $\beta(t)=e^{tz_n}Pe^{-tz_n}$ then 
\begin{align*}
\dist_{\infty}(P,Q_n) & \le L_0^1(\beta)=\|z_nP-Pz_n\|\le \max\{\|Pz_n(1-P)\|,\|(1-P)z_nP\|\}\\
& \le \|z_n\|\le \pi/2<\|v_0\|=L_0^{t_0}(\gamma).
\end{align*}
Since $Q_n\to Q$, we are done.
\end{proof}

\begin{rem}[Cut locus]\label{nohayejem}
See also Corollary \ref{corocl} below for another particular case. A related question is about the rectifiable diameter of $Gr(P_0)$; this relation is patent in the Riemanian setting because the cut locus of $P$ consists exactly of the points that are either conjugate to $P$ or the points such that there exist two minimizing geodesics from $P$ arriving at the point). The rectifiable diameter of $Gr(P_0)$ is a relevant invariant of $C^*$-algebras, related to the \textit{exponential length}, see for instance \cite{ncphil} and the references therein. In particular it is shown there that if $\mathcal A$ has real rank zero and the cancellation property, then in fact the diameter of $Gr(P_0)$ is $\frac{\pi}{2}$ \cite[Theorem 3.2]{ncphil}. It is also worth mentioning here that if $\mathcal A$ is  purely infinite simple then this rectifiable diameter is exactly $\pi$ \cite[Theorem 3.3]{ncphil}.
\end{rem}

Now we consider the enveloping von Neumann algebra of $\mathcal A$, let $P\wedge P'$ denote the infimum of the projections (again a projection) and let $\sim$ denote the Murray-von Neumann equivalence of projections. Then a full characterization of points that can be joined with a geodesic was obtained by Andruchow in \cite{andru1} as follows:
\begin{teo}Let $P,Q\in Gr(P_0)$ with $\mathcal A$ a von Neumann algebra. Then there exists a geodesic joining $P,Q$ if and only if 
$$
P\wedge (1-Q)\sim Q\wedge (1-P).
$$
In this case there exist a minimizing geodesic joining them. Moreover, the geodesic is unique if and only if $P\wedge (1-Q)=0$.
\end{teo}
In finite dimensional algebras, or in algebras of compact operators, the condition is automatically fulfilled, see \cite{andru1}. If $\mathcal A$ is not a von Neuman algebra, and the condition is fullfilled, the geodesic might not have speed in $\mathcal A$, so some caution is required. But by considering the enveloping von Neumann algebra of the $C^*$-algebra $\mathcal A$, it follows that if there exist a geodesic joining $P,Q$, then the condition must be fullfilled. That the condition is expressed in terms of the weak closure of $\mathcal A$ is a phenomena that will reappear soon (Remark \ref{polard} and Theorem \ref{mononcon} below).  We now discuss uniqueness of geodesics, and show that before the first cut locus they are unique as in the Riemannian setting.

\begin{teo}\label{cl} Let $V\in T_PGr(P_0)$, let $\gamma$ be  the unique geodesic from $P$ with initial speed $V$. Then
\begin{enumerate}
\item If $t\|V\|<\pi/2$, then the only minimizing geodesic joining $P,Q=\gamma(t)$ is $\gamma$.
\item If $t_0\|V\|=\frac{\pi}{2}$ and either $\pm\frac{\pi}{2}$ is an eigenvalue of $t_0 V$, then there is another minimizing geodesic $\gamma_1\subset \mathcal A''$ joining $P$ to $Q=\gamma(t_0)$. Moreover, if the eigenvalue is isolated, then $\gamma_1\subset \mathcal A$ and $\gamma$ is not minimizing past $t_0$.
\end{enumerate}
\end{teo}
\begin{proof}
 First we prove the second assertion. To simplify the notation, after rescaling $V=[v,P]$ we can assumme that $t_0=1$. By Lemma \ref{campos}(2) we have that both $\pm i\pi/2\in \sigma(v)$. Let $\xi_1\in\mathcal H$ of unit norm be such that $v\xi_1=i\frac{\pi}{2}\xi_1$, let $p_+=\xi_1\otimes \xi_1$ then $p_+\in\mathcal A''$. Let $\xi_2=\mathfrak s_P\xi_1$, then  $p_-=\xi_2\otimes\xi_2=\mathfrak s_P p_+\mathfrak s_P\in\mathcal A''$ and since $\mathfrak s_pv=-v\mathfrak s_p$ (Lemma \ref{cl})  $vp_-=-i\frac{\pi}{2}p_-$. Now $v^2p_{\pm}=-\frac{\pi^2}{4}p_{\pm}$ hence $|v|p_{\pm}=\frac{\pi}{2}p_{\pm}$ and $up_{\pm}=\pm ip_{\pm}$. Eigenvalues corresponding to different eigenvectors are orthogonal, hence $p_+p_-=0$. Write
$$
|v|=|v|(1-(p_++p_-))+|v|(p_++p_-)=|v|(1-(p_++p_-))+\frac{\pi}{2}(p_++p_-)
$$
then
$$
v=u|v|=u|v|(1-(p_++p_-))+i\frac{\pi}{2}(p_+-p_-)=v_{\perp}+i\frac{\pi}{2}(p_+-p_-),
$$
with $v_{\perp}p_+=0=v_{\perp}p_-$ and $\|v_{\perp}\|\le \pi/2$. Now 
$$
\mathfrak s_P v(1-p_++p_-)=-v \mathfrak s_P(1-(p_++p_-))=-v (1-(p_++p_-))\mathfrak s_P
$$
and
$$
\mathfrak s_P (p_+-p_-)=\mathfrak s_Pp_+-p_+\mathfrak s_P =-(p_+-p_-)\mathfrak s_p
$$
hence both $v_{\perp}, i\frac{\pi}{2}(p_+-p_-)\in \widetilde{\mathcal C}_P$. Hence we can consider 
$$
v_1=v_{\perp} -\frac{i\pi}{2}(p_+-p_-)\in \mathcal A'',
$$
which is skew-adjoint and $P$-codiagonal, with $\|v_1\|=\|v\|=\pi/2$. Note also that 
$$
\exp(2v)=-p_+-p_-+e^{2v_{\perp}}(1-(p_++p_-))=\exp(2v_1).
$$
Let $\gamma_1(t)=e^{tv_1}Pe^{-tv_1}\subset \mathcal A''$, then by (\ref{anticonmexp}) we have that 
$$
2\gamma(1)-1=e^{v}2Pe^{-v}-1=e^{v}\mathfrak s_P e^{-v}=e^{2v}\mathfrak s_P=e^{2v_1}\mathfrak s_P=2\gamma_1(1)-1.
$$
This shows that $\gamma(1)=\gamma_1(1)$ and since both geodesics have the same speed, they are both minimizing joining $P,Q$. Assumme now that $i\pi/2$ is an isolated eigenvalue, then by the symmetry of the spectrum so is $-i\pi/2$. Let $p_+$ be the full eigenprojection, which is now in $\mathcal A$.  Since $p_+=\mathfrak s_p \mathfrak s_p p_+=\mathfrak s_p p_-$, we see that $p_-$ is in fact the eigenprojection for the opposite eigenvalue, and it also follows that $\|v_{\perp}\|<\frac{\pi}{2}$. So let  $\delta=\frac{\pi}{2}-\|v_{\perp}\|>0$, and take any $\varepsilon>0$ such that $\varepsilon<\frac{\delta}{\pi-\delta}$. We claim that $\gamma$ is not minimizing in $[0,1+\varepsilon]$. Let
$$
v_2=(1-\varepsilon)i\frac{\pi}{2}(p_--p_+)+(1+\varepsilon)v_{\perp},
$$
then $v_2$ is $P$-codiagonal and $\|v\|_2=\max\{(1-\varepsilon)\frac{\pi}{2},(1+\varepsilon)\|v_{\perp}\|\}=(1-\varepsilon)\frac{\pi}{2}$ by our choice of $\varepsilon$. Let $\gamma_2$ be the geodesic from $P$ with initial speed $V_2=[v_2,P]$; a straightforward computation shows that $e^{2v_2}=e^{2(1+\varepsilon)v}$. Hence again using the trick of the symmetry $\mathfrak s_P$ we see that $\gamma_2(1)=\gamma(1+\varepsilon)$. But then
$$
\dist_{\infty}(P,\gamma(1+\varepsilon))\le L_0^1(\gamma_2)=\|v_2\|=(1-\varepsilon)\frac{\pi}{2} < (1+\varepsilon)\frac{\pi}{2}=L_0^{1+\varepsilon}(\gamma)
$$
thus $\gamma$ is not minimizing in $[0,1+\varepsilon]$. Finally,  we prove the first assertion: if $t\|V\|=t\|v\|<\pi/2$, then $\|2tv\|<\pi$. We have that $\gamma(t)=\gamma_1(t)$  for some geodesic $\gamma_1(t)=e^{tv_1}Pe^{-tv_1}$ with $\|v_1\|=\|v_0\|$ (by repeating the argument with the symmetry $\mathfrak s_P$) is only possible if $e^{2tv}=e^{2tv_1}$. But then by the injectivity of the exponential map in exponents of norm strictly less than $\pi$, we see that it must be $v_0=v_1$ hence $\gamma_1=\gamma$.
\end{proof}

\begin{rem}[Real case]  In the case of the real Grassmannians  the proof needs some adaptation: let $e_1=\mathfrak{Re}(\xi_1)$ and $e_2=\mathfrak{Im}(\xi_1)$, then $ve_1=-e_2$ and $ve_2=e_1$ so $ue_1=-e_2$ and $ue_2=e_1$ and we can write
$$
v=v_{\perp}+\frac{\pi}{2}(-e_1\otimes e_2+e_2\otimes e_1).
$$
Then we change the sign of the second term to obtain $v_1$ and the rest of the proof follows in the same fashion.
\end{rem}

\begin{coro}\label{corocl} Assume $P_0$ has finite rank or co-rank (in particular, any finite dimensional Grassmannian). If $P,Q\in Gr(P_0)$ and $\dist(P,Q)=\pi/2$, then there exist at least two minimizing geodesics joining them, and unit speed geodesics are not minimizing past $\pi/2$.
\end{coro}
\begin{proof}
If $P_0$ has finite rank or co-rank, the same goes for any $P\in Gr(P_0)$. That there exist a geodesic in this case is known, see for instance \cite[Section 3]{andru1}. Now if $V$ is described as in (\ref{VX}), in both cases we have $\lambda$ and $\lambda^*$ compact operators (one of them has a finite dimensional domain, the other a finite dimensional co-domain). Thus $v$ (equivalently, $V$)  is a compact operator and $\pm i \frac{\pi}{2}$ are isolated eigenvalues, and the previous theorem ends the proof.
\end{proof}

\subsection{Conjugate points}\label{scpoints}

In this section we characterize the differential of the exponential map and its invariant subspaces, and we present all possible candidates for tangent conjugate points to $P\in Gr(P_0)$ along a geodesic $\gamma(t)=e^{tv}Pe^{-tv}$.

\begin{rem}[Differential of the exponential map of the connection]\label{difexp}
Since the exponential map of the reductive connection is $Exp_P(V)=e^{[V,P]}Pe^{-[V,P]}$, one can compute its differential explicitly, or use the Theorem  \ref{JF} and the well-known fact that if $\mu$ is the unique Jacobi field along $\gamma(t)=Exp_P(tV)$ with $\mu(0)=0$ and $D_t\mu(0)=\dot{\mu}_0=Y$, then we have
$$
D(Exp_P)_V(Y)=\mu(1)=e^v[\sinhc (\ad v)y,P]e^{-v}
$$
where as before, $V=[v,P]$ and $Y=[y,P]$ with $v,y\in \widetilde{\mathcal C}_P=\A_{sk}\cap \C_P$. This formula was obtained for any symmetric Banach space in \cite[Lemma 3.10]{neeb}.
\end{rem}

\begin{rem}[Complexification and the spectrum]\label{complexifi} Since we are dealing with real Banach spaces such as $\mathcal A_{sk}$, it will be convenient to clarify the notions of spectrum (and the notation) we will be using. We recall that if $X$ is a real Banach space and $T\in \mathcal B(X)$ a real linear bounded operator, then with the standard complexification $X'=X\oplus i X$ and $T'=T\oplus T\in\mathcal B(X')$ it is well-known that
$$
\sigma(T')\cap\mathbb R=\sigma_{\mathbb R}(T)\; \left(\;=\{s\in\mathbb R: T-s1\; \textrm{ is not invertible}\}\;\right). 
$$
We further note that $\mathcal A=\mathcal A_{sk}\oplus i\mathcal A_{sk}$ is a standard complexification of $\mathcal A_{sk}$. Hence if $T'\in \mathcal B(\mathcal A)$ has $\mathcal A_{sk}$ as an invariant subspace, we can conclude that 
$$
\sigma(T'|_{\mathcal A_{sk}})=\sigma(T')\cap \mathbb R.
$$
In particular $T'$ is invertible in $\mathcal A$ if and only if it is invertible restricted to $\mathcal A_{sk}$. The same remark holds if we replace $\mathcal A_{sk}$ with the self-adjoint operators.
\end{rem}

\begin{rem}[Factorization of the differential of the exponential map] \label{facto}
For $v\in\widetilde{\mathcal C}_P$, the operator $\ad v$ does not preserve that space. However the operator $\ad^2 v$ does. Thus if we pair the roots of the entire function $\sinhc$ and their opposites, the Weierstrass factorization of that function allows us to write 
\begin{equation}\label{wf}
\sinhc(z)=\prod_{k\ne 0} \left( 1+ \frac{z}{i k\pi }\right)= \prod_{k\ge 1} \left( 1+ \frac{z^2}{k^2\pi^2 }\right)\qquad z\in\mathbb C
\end{equation}
By means of the holomorphic functional calculus we obtain
$$
\sinhc(t\ad v)=\prod_{k\ne 0} \left( 1+ \frac{t\ad v}{i k\pi }\right)= \prod_{k\ge 1} \left( 1+ \frac{t^2\ad^2 v}{k^2\pi^2 }\right),
$$
and in the last expression we have as building blocks linear operators from $\widetilde{C}_P$ into itself. 
\end{rem}

\begin{rem}[The spectrum and eigenvalues of elementary operators]\label{speo}
Note that if $0\ne s\in \sigma(\ad^2v)$ is an eivenvalue in $\mathcal A$, then it is also an eigenvalue in $\mathcal C_P$: because if $0\ne x\in \mathcal D_P$ is such that $\ad^2v(x)=\lambda x$, then applying $\ad v$ we see that 
$$
\ad^2v [v,x]=s [v,x]
$$
and since $[v,x]\in \mathcal C_P$ and it is nonzero (because $sx \ne 0$) we see that $s$ is an eigenvalue of $\ad^2v|_{\mathcal C_P}$.  However, eigenvectors are not our only source of problems. In the infinite dimensional setting, it is possible for a bounded linear operator to be injective but not bounded from below, or injective but not surjective (the range can be a proper subspace or a dense subspace). It was shown by Lumer and Rosenblum \cite[Theorem 10]{rosen} that when $\mathcal A=\mathcal B(\mathcal H)$, we have $\sigma(\ad v)=\{s-t: s,t\in\sigma(v)\}$ if we consider $\ad v$ as an operator on $\mathcal A$. This equality of sets for the spectrum was later extended to elementary operators in von Neumann factors and also for prime and primitive $C^*$-algebras by Mathieu in \cite[Section 4]{mathieu} and \cite[Theorem 3.9]{mathieu2}; see also \cite{abrams}. In \cite[Theorem 5]{rosen} it is also shown that for any complex Banach algebra the inclusion $\subset$ holds, while the equality might fail. 
\end{rem}

\begin{rem}[Monoconjugate and epiconjugate points] We will identify those points $V_0$ where our operator $D(\Exp_P)_{V_0}$ is not invertible. These points will be called \textit{tangent conjugate points} to $P$. If the operator is not injective, it is customary to call the point \textit{monoconjugate}, and the (real) dimension of its kernel is the \textit{order} of the conjugate point. If the operator is not surjective the point is called \textit{epiconjugate}. This phenomena on conjugate points was first observed in the Riemann-Hilbert setting by Grosmann \cite{gros} and McAlpin \cite{mcalpin}. 
\end{rem}


Let $P\in Gr(P_0)$, let $V=[v,P]\in T_P Gr(P_0)$ with $v^*=-v\in \C_P$ and $\|V\|=1$. Let $\gamma(t)=Exp_P(tV)=e^{tv}Pe^{-tv}$. 
\begin{lem}\label{fcp} The candidates to tangent conjugate points are at located at $TV$ with
$$
T(k,s,s')=\frac{k\pi}{|s-s'|},\qquad k\in\mathbb Z^*=\mathbb Z\setminus\{ 0\}, \qquad s\ne s'\in \sigma(V)\subset [-1,1].
$$ 
\end{lem}
\begin{proof}
Since $x\mapsto [x,P]$ is an isomorphism from $\widetilde{\mathcal C}_P$ onto $T_PGr(P_0)$, by Remark \ref{difexp} it suffices to study the operator $x\mapsto \sinhc(t\ad v)x$. If we look at the final factorization of the previous remark, it is apparent that each building block $1+ \frac{t^2\ad^2 v}{k^2\pi^2}$ will fail to be invertible at $t$ such that
$$
-k^2\pi^2 \in t^2 \sigma(\ad^2v)=t^2\sigma(\ad v)^2
$$ 
for each $k\ne 0$. If we consider $x\mapsto \ad v (x)=[v,x]$, but as an operator from $\mathcal A$ into itself, we have that  $\sigma(\ad v)\subset\{s-s':s,s'\in \sigma(v)\}$, by \cite[Theorem 10]{rosen}. Hence $\sigma(\ad^2 v)\subset\{-|s-s'|^2: s,s'\in \sigma(V)\} \subset [-4,0]$ since $\sigma(V)=i\sigma(v)\subset i[-1,1]$. If $s=s'$ it must be $k=0$ which is excluded. Hence it must be $-k^2\pi^2 =-t^2 |s-s'|^2$ for some $s\ne s'$ in $\sigma(V)$, and this proves the lemma.
\end{proof}

\begin{rem}[Building blocks]\label{bloque} For each $T=T(k,s,s')$ the point $Q=Exp_P(TV)$ is conjugate to $P$ when $\sinhc(T\ad v)$ is not invertible, and this happens if and only if any of the operators
$$
x\mapsto (\ad^2 v+\frac{j^2}{k^2}|s-s'|^2 1)x,\quad j\in\mathbb Z^*
$$
is not invertible in $\mathcal C_P$. Equivalently, naming $\mu_j=\frac{|j|}{|k|}|s-s'|>0$, when any of the operators $\ad^2 v+\mu_j^21$ is not invertible. This will happen if and only if there exists $s_1\ne s_2\in\sigma(V)$ such that $|s_1-s_2|=\mu_j$.
\end{rem}

\begin{rem}Consider the case of the projective real line, presented as the orthogonal group-orbit of a one-dimensional projection. Since the tangent space is one-dimensional, we have  $\ad^2v\equiv 0$ there, hence
$$
D(Exp_P)_V(Y)=\mu(1)=e^v[\sinhc (\ad v)y,P]e^{-v}=e^v[y,P]e^{-v}=Y,
$$
the differential of the exponential map is always the identity map, and \textit{there are no conjugate points at all}. For finite dimensional Grassmannians however, it is known (and we will show below) that this is the only case with no conjugate points.
\end{rem}

\begin{rem}[First tangent conjugate point]\label{ftcp}
When  $\|V\|=1$ we know that $1$ or $-1$ belongs to $\sigma(V)$, but by Lemma \ref{campos}(2) we have both $-1,1$ belong to $\sigma(V)$. Hence the first (candidate to) tangent conjugate point occurs at $t_0=\pm\frac{\pi}{2}$, since it corresponds to $k=\pm1$, $s=1, s'=-1$. In this case note that the corresponding operator of Remark \ref{bloque} will fail to be invertible if there exist $s_1,s_2\in \sigma (V)$ such that $|j|=|k||s_1-s_2|/2$; but since $|s_1-s_2|\le 2$ then it must be $|j|\le |k|=1$. Hence the only operator of interest in this case is 
$$
 x\mapsto (4+\ad^2 v)x\qquad  x\in \widetilde{\mathcal C}_P
$$
which corresponds to $\mu=2$. We will show that this operator is never invertible (in any $C^*$-algebra) hence $Q=\gamma(\frac{\pi}{2})$ is always the first conjugate point along $\gamma$.
\end{rem}

\begin{rem}[Block operators]\label{bloo} We represent $P$ as a block matrix and likewise generic tangent vectors at $P$ by $V=\lambda +\lambda^*$, $X=\rchi+\rchi^*$ as in equation (\ref{VX}) above, with $\|V\|=1$ (note that $\lambda=PV(1-P)=PV$ and likewise $\rchi=PX(1-P)=PX$). As before, we let $x=[X,P]$ and $v=[V,P]$. The operators of Remark \ref{bloque} have the form
$$
x\mapsto [v,[v,x]]+\mu^2 x=v^2x+xv^2-2vxv+\mu^2 x
$$
for some $\mu>0$, i.e. $\mu=|j||k|^{-1}|s-s'|$ for $s\ne s'\in \sigma(V)$ and $j,k\in\mathbb Z^*$. Moreover since we are only interested in the case when this operator is not invertible, it must be 
$$
\mu^2\in\sigma(-\ad^2 v)\subset\{|s_1-s_2|^2:s_1\ne s_2\in\sigma(V)\}.
$$
In particular we know then that $0<\mu\le 2$. The operator can then be rewritten in terms of block-operators as
\begin{equation}\label{opermu}
\rchi\mapsto |\lambda^*|^2\rchi+\rchi|\lambda|^2-2\lambda\rchi^*\lambda-\mu^2\rchi.
\end{equation}
\end{rem}

\begin{rem}[Polar decomposition of the speed]\label{polard} We assume $\mathcal A$ embedded in a concrete operator algebra $\mathcal B(\mathcal H)$. Since $\mathrm{rank}(|\lambda|)=\mathrm{rank}(|\lambda^*|)$, we can consider the partial isometry $\Omega: \ran(1-P)\to \ran(P)$ given by $\Omega|\lambda|\xi=\lambda\xi$ on $\ran(|\lambda|)$ (and its closure), and extended by zero on $\ran(|\lambda|)^{\perp}\subset\ran 1-P$. That $\Omega$ is well-defined, bounded and a partial isometry can be proved in the same fashion as in the case of the standard polar decomposition for an operator on  a fixed Hilbert space. Apparently, we have $\lambda=\Omega|\lambda|$ and $\Omega^*\Omega=P_{|\lambda|}$, where the latter is the projection onto the closure of the range of $|\lambda|$.  Moreover $\Omega\Omega^*=P_{|\lambda^*|}$ and $\Omega |\lambda|\Omega^*=|\lambda^*|$. 

\smallskip

Let $V=U|V|$ be the polar decomposition of $V$ in the enveloping von Neumann algebra of $\mathcal A$, then $|V|=|v|$ and $\Omega=PU(1-P)$. In particular $v=u|v|=u|V|$ with
\begin{equation}\label{uU}
U=\left(\begin{array}{cc} 0 & \Omega  \\ \Omega^* & 0 \end{array}\right) \quad \textrm{ and } \quad u=\left(\begin{array}{cc} 0 & \Omega  \\ -\Omega^* & 0 \end{array}\right).
\end{equation}
Note that  $\Omega\in \mathcal A$ if and only if $U\in\mathcal A$. This happens in particular if $V$ is a \textit{regular} element of $\mathcal A$ (i.e. $0$ is isolated in $\sigma(V)$), see \cite[Theorem 2.6]{arias}.
\end{rem}

\begin{defi}\label{plam} We will use ${\mathfrak{Re}},\mathfrak{Im}:\mathcal A\to\mathcal A$ to denote the real linear operators that take the symmetric and skew symmetric part of an operator in $\mathcal A$, i.e $\mathfrak{Re}(x)=(x+x^*)/2$ and $\mathfrak{Im}(x)=(x-x^*)/2$. For fiven $V\in \mathcal A$, we will decompose $\mathcal A_{sk}$ in a direct sum of three subspaces:
\begin{equation}\label{ds}
\mathcal A_{sk}=P_v\mathcal A_{sk} P_v\bigoplus\; \widetilde{\mathcal C}_{P_v}\; \bigoplus (1-P_v)\mathcal A_{sk}(1-P_v),
\end{equation}
where
$$
\widetilde{\mathcal C}_{P_v}= P_v\mathcal A_{sk}(1-P_v)\oplus (1-P_v)\mathcal A_{sk}P_v=P_v\mathcal A_{sk}\oplus \mathcal A_{sk}P_v.
$$
Denoting $P_{|\lambda^*|}\le P$ the range projection of $\lambda\lambda^*$ and $P_{|\lambda|}$  the range projection of $\lambda^*\lambda$, we have we 
$$
P_V=P_{|V|}=P_v=P_{|\lambda^*|}+P_{|\lambda|}.
$$
\end{defi}

\begin{prop}\label{bloques} Let $V=[v,P]\in T_PGr(P_0)$. Then each of the three subspaces in (\ref{ds}) are invariant for $\sinhc(T\ad V)$. If $V$ has unit norm,  $v=u|v|$ is the polar decomposition of $v$ and $T=T(k,s,s')$ as in Lemma \ref{fcp}, then with respect to this direct sum we have 
$$
\sinhc(T\ad v)=L_u(\Pi_-\oplus \Pi_+)L_{u^*} \bigoplus (L_{\sinhc(Tv)}+R_{\sinhc(Tv)}-1)\bigoplus 1
$$
where  $\Pi_-,\Pi_+\in\mathcal B(\mathcal C_P)$ are given by 
$$
\Pi_-=\Pi_{j\in\mathbb N}\big(1-\frac{1}{\mu_j^2}(L_{|v|}-R_{|v|})^2\big)\mathfrak{Re}
\qquad \Pi_+=\Pi_{j\in\mathbb N}\big(1-\frac{1}{\mu_j^2}(L_{|v|}+R_{|v|})^2\big)\mathfrak{Im}
$$
for $\mu_j=j|k|^{-1}|s-s'|$, and they preserve self-adjoint (resp. skew-adjoint) operators. 
\end{prop}
\begin{proof}
Using that $uv=vu=-|v|$ we have that 
$$
L_{u^*}(1+ \ad^2 v)x=u^*x-|v|^2u^*x-u^*x+2|v| (u^*x)^*|v|
$$
where in the last term we used that $(u^*x)^*=x^*u=-xu$ since $x^*=-x$.  If we split $u^*x=a+b$ in its self-adjoint part $a$ and its skew adjoint part $b$, we obtain
$$
a-|v|^2a-a|v|^2+2|v|a|v|   +  b-|v|^2b-b|v|^2-2|v|b|v|,
$$
where the sum of the terms with $a$ is self-adjoint while the sum of the terms with $b$ is skew-adjoint. Then we have shown that
$$
L_{u^*}(1+ \ad^2 v)=(1- (L_{|v|}-R_{|v|})^2)\mathfrak{Re}\,L_u^*+ (1- (L_{|v|}+R_{|v|})^2)\mathfrak{Im}\,L_u^*.
$$
Note that both $L_{|v|}\pm R_{|v|}$ preserve the space of self-adjoint and the space of skew-adjoint elements of $\mathcal A$. Then if we multiply by $(1+\frac{1}{\mu_2^2}\ad^2v)$ on the right, and apply the identity above but for $\mu_2$ instead of $\mu_1$, we have
\begin{align*}
L_{u^*}\left(1+ \frac{\ad^2 v}{\mu_1^2}\right)\left(1+ \frac{\ad^2 v}{\mu_2^2}\right) =& \left(1- \frac{(L_{|v|}-R_{|v|})^2)}{\mu_1^2}\right) \left(1- \frac{(L_{|v|}-R_{|v|})^2)}{\mu_2^2}\right)\mathfrak{Re}\, L_{u^*}
\\
 & + \left(1- \frac{(L_{|v|}+R_{|v|})^2}{\mu_1^2}\right)\left(1- \frac{(L_{|v|}+R_{|v|})^2}{\mu_2^2}\right)\mathfrak{Im}\, L_{u^*}.
\end{align*}
Let 
$$
S_j=(1-\frac{1}{\mu_j^2}\left(L_{|v|}-R_{|v|}\right)^2)\qquad\qquad T_j=(1-\frac{1}{\mu_j^2}(L_{|v|}+R_{|v|})^2),
$$
then $\Pi_-=\Pi_{j\in\mathbb N} S_j\mathfrak{Re}$ and $\Pi_+=\prod_{j\in\mathbb N}T_j\mathfrak{Im}$. Iterating the argument above for the products of $1+\frac{\ad^2 v}{\mu_j^2}$, it is plain by Remark \ref{facto} and Lemma \ref{fcp} that 
$$
L_{u^*}\sinhc( T\ad v)=\Pi_- L_{u^*}+\Pi_+L_{u^*} \textrm{ in } \mathcal {\mathcal C}_P.
$$
Since $L_{P_v}=L_u^*L_u$ we obtain 
$$
L_{P_v}\sinhc( T\ad v)=L_u(\Pi_-\oplus \Pi_+)L_{u^*} \textrm{ in } \mathcal {\mathcal C}_P.
$$
On the other hand, a straightforward computation shows that 
$$
(1-P_v)(1+\frac{\ad^2 v}{\mu_j^2})x=(1-P_v)x(1+\frac{v^2}{\mu_j^2}),
$$
hence $(1-L_{P_v})\sinhc(T\ad v)=(1-L_{P_v})R_{\sinhc(Tv)}$. Hence
\begin{align}
\sinhc(T\ad v)x & =\sinhc(T\ad v)L_{P_v}x+\sinhc(T\ad v)(1-L_{P_v})x \nonumber \\
\sinhc(T\ad v)x &=u(\Pi_-\oplus\Pi_+)(u^*x)+(1-P_v)x\sinhc(Tv).\label{parc}
\end{align}
Now let $x=P_vxP_v$, then the second term in (\ref{parc}) vanishes, since $(1-P_v)x=0$. For the first term we have
\begin{align*}
u\Pi_{\pm}(u^*x) & =u\Pi_{\pm}(u^*xP_v)=u\Pi_{\pm}(R_{P_v}(u^*x))=uR_{P_v}\Pi_{\pm}(u^*x)\\
&=u\Pi_{\pm}(u^*x) P_v\in P_v\mathcal A P_v
\end{align*}
since $u=P_vu$. This settles the assertion for the first invariant subspace. 

Now if $x=(1-P_v)x(1-P_v)$, it is plain that $u^*x=0$ therefore the first term vanishes.  It is easy to see that  $(1-P_v)\sinhc(Tv)=1-P_v$, thus the second term equals $(1-P_v)x(1-P_v)$ i.e., the space is invariant and $\sinhc(\ad T v)$ is the identity operator there. 

Lastly, let $x=xP_v+P_vx$ i.e. $x^*=-x$ is $P_v$-codiagonal. Then $P_vxP_v=0$ and $vxv=0$, hence $\ad^{2k}v x=v^{2k}x+xv^{2k}$ for $k\ge 1$ (in particular the operator $\sinhc(T\ad v)$ preserves $P_v\mathcal A_{sk}\oplus \mathcal A_{sk}P_v$). A straightforward computation throws 
\[
\sinhc(T\ad v)x  =\sinhc(Tv) x+x\sinhc(Tv)-x. 
\]
\end{proof}

\subsection{The kernel of $D\Exp_P$} In this section we examine the nontrivial componentes of the kernel, we begin with the codiagonal part.

\begin{rem}[$P_V$-codiagonal tangent locus]\label{pvc} Solutions in the space $\widetilde{\mathcal C}_{P_V}$:
\begin{enumerate}
\item[i)]  This space in the above decomposition will be trivial when $P_V=1$, and this is equivalent to $P_{|\lambda^*|}=P$, $P_{|\lambda|}=1-P$ (notation as in Definition \ref{plam}). Since $P_1=\Omega P_2 \Omega^*$ are conjugated, this is only possible when $P,1-P$ have the same rank. 
\item[ii)]Let  $S=\sinhc(Tv)$. Then  $x\in \widetilde{\mathcal C}_{P_V}$ belongs to the kernel of $\sinhc(T\ad v)$ if and only if 
\begin{equation}\label{sxpv}
Sx=xP_v.
\end{equation}
It is plain that such $x$ is in the kernel. Here is why any solution satisfies (\ref{sxpv}):  by the previous proposition we have  $Sx+xS=x$, then $P_vSx+P_vxS=P_vx$. Now $P_v$ commutes with $S$ and moreover $P_vS=P_v(1+v+\dots)$ therefore $P_vSx=P_vs$ since $P_vxP_v=0$. Hence $P_vS x =0$ and 
\begin{align*}
Sx & =S(1-P_v)x+SP_vx=s(1-P_v)x=(1+v+\dots)(1-P_v)x\\
&=(1-P_v)x=x-P_vx=xP_v.
\end{align*} 
\item[iii)] The equation (\ref{sxpv}) has nontrivial solution if and only if there  exist $j\in \mathbb N$ and $0\ne s_0\in \sigma(|v|)$ such that  $Ts_0=j\pi$, or equivalently
\begin{equation}\label{Tj}
j\, |s-s'|=|k| s_0.
\end{equation}
Here we explain why and characterize the solutions: we have $S=P_{\ker v}+S'$ for some $S'$ with $|S'|=|S-P_{\ker v}|\le 1-\varepsilon<1$.  The operator $S'$ is obtained by means of $\sinhc$ applied to $TV$ with the eigenprojections of $TV$ corresponding to $j\pi$ (which are mapped to $0$) removed. Note that $P_{\ker v}=1-P_v$, then from $Sx=xP_v$ we have $(1-P_v)x+ S' x=xP_v$, which is equivalent to $S'x=0$. Thus the solutions $x$ are characterized by the equation $P_0x=0$, where $P_0$ is the spectral projection of $Tv$ corresponding to the points of the spectrum of $Tv$ that are different of $j\pi$, $j\ge 1$. In particular if $\sigma(TV)\subset \pi\mathbb Z$ then $S'=0$ and the solution is the whole space $\widetilde{\mathcal C}_{P_V}$.
\item[iv)] For the first conjugate point, the condition (\ref{Tj}) is impossible since in that case $k=1$ and $|s-s'|=2$ (Remark \ref{ftcp}). So it must be  $2|j|=s_0\le 1$ which is not possible. Then we conclude that \textit{there are no $P_V$-codiagonal solutions} for the first conjugate locus, and \textit{all solutions are in} $P_v\mathcal A P_v$.
\end{enumerate}
\end{rem}

We now proceed to study $\sinhc(T\ad v)$ in the final block  $P_v\mathcal AP_v$. In this space, $u$ acts like a unitary operator, to be precise:

\begin{lem}\label{luu}
The map $L_{u^*}:P_v\widetilde{\mathcal C}_PP_v\longrightarrow \mathcal D_0=\{Z\in \mathcal D_P\cap \mathcal D_{P_v}: Z^*=uZu^*\}$ is an isometric isomorphism of Banach spaces, with inverse $L_u$.  For each $Z$ in the later space we have $PZ=\Omega^*\rchi=ZP$ and $(1-P)Z=\Omega\rchi^*=Z(1-P)$, where $\rchi,\rchi^*$ are the components of $x\in \widetilde{\mathcal C}_P$ as in (\ref{VX}). Moreover for each $Z$ we have
\begin{equation}\label{zi}
ZP=PZ=\Omega^*(1-P)Z\Omega \quad \mathrm{ and }\quad (1-P)Z=\Omega PZ \Omega^*.
\end{equation}
\end{lem}
\begin{proof}
Let $x\in P_v\widetilde{\mathcal C}_PP_v$, let $Z=u^*x$. Then $Z^*=x^*u=xu^*=uu^*xu^*=uZu^*$.  We have $P_vZ^*=P_vuZu^*=uZu^*=Z^*$ and likewise $Z^*P_v=Z^*$, hence $Z\in \mathcal D_{P_v}$. Now $PZ=Pu^*x=P(\Omega \rchi^* +\Omega^*\rchi)=\Omega\rchi^*=u^*x P=ZP$, this proves that $Z\in \mathcal D_P$ and the assertion on $Z$ and the fact that $Z\in \mathcal D_P$. It is plain that $L_u^*$ is isometric since it is a partial isometry and we are restricting it to its support. Now assume that $Z\in \mathcal D_0$, then decomposing $Z=PZ+(1-P)Z$ and using $uZu^*=Z$ immediately shows the validity of (\ref{zi}). Let  $x=uZ$. Then  $Z^*u=uZP_v=uP_vZ=uZ$ hence $x^*=Z^*u^*=uZu^*u^*=-uZ=-x$. Moreover $P_v x=uZ=x$, $xP_v=-x^*P_v=-x^*=x$, thus $x\in P_v\mathcal A_{sk}P_v$. Finally, by writting $Z$ a a diagonal block matrix $Z=\tilde{z}\oplus z$ with respect to $P$, and likewise with $u$ as in (\ref{uU}), we see that
$$
Px=Puz=-\Omega z =uz(1-P)=x(1-P),
$$
thus $x\in \mathcal C_P$. 
\end{proof}

\begin{rem}\label{muu} By Proposition \ref{bloques} and Lemma \ref{luu}, the operator $\sinhc(T\ad v)$ restricted to $P_v\mathcal A P_v$ is equivalent to the operator $\Pi_-\oplus \Pi_+$ acting in $\mathcal D_0$. Let
$$
\mathcal A_0=(1-P)\mathcal D_0 (1-P)\simeq P\mathcal D_0 P=P_{|\lambda*|}\mathcal A P_{|\lambda*|}\simeq P_{|\lambda|}\mathcal A P_{|\lambda|}
$$
the isomorphisms implemented by $\Ad_{\Omega}$. By the same Lemma \ref{luu}, it suffices to study the restriction of this operator $\Pi_-\oplus \Pi_+$ to $\mathcal A_0$. So if we name 
$$
z=\Omega^*\rchi \in \mathcal A_0,
$$
we have that $Z\in \mathcal D_0$ can be written as $Z=\Omega^* z\Omega + z$, where the first summand is in $P\mathcal D_0 P$. For this restriction of the operator we are considering, it is plain that we can replace $L_{|V|}$ with $L_{|\lambda|}$ and $R_{|V|}$ with $R_{|\lambda|}$, since $|V|=|\lambda|$ in $(1-P)\mathcal A (1-P)$.   For brevity we will denote $L=L_{|\lambda|}$, $R=R_{|\lambda|}$ from now on.

Now we write $z=a+b$ with $a^*=a$, $b^*=-b$ and we notice that the building blocks of $\Pi_-\oplus \Pi_+$ are given  (modulo a nonzero factor $\mu^2$) by the operators
\begin{align}
a\longmapsto |\lambda|^2a+a|\lambda|^2-2|\lambda|a|\lambda|-\mu^2 a &=\left((L-R)^2-\mu^2  1\right )a \label{1}\\
b\longmapsto |\lambda|^2b+b|\lambda|^2+2|\lambda|b|\lambda|-\mu^2 b &=\left((L+R)^2-\mu^2  1\right )b\label{2}.
\end{align}
Note that 
$$
\left((L+ R)^2-\mu^2  1\right )=(L+R-\mu  1 )(L+ R+\mu  1 )
$$
and that both factors preserve the skew-adjoint operators. Likewise
$$
\left((L-R)^2-\mu^2  1\right )=(L-R-\mu  1 )(L- R+\mu  1 )
$$
but \textit{neither factor preserves the self-adjoint operators}.  Let $\Lambda=\Lambda(k,s,s')$ denote the index set
$$
\Lambda=\{j\in\mathbb N: \exists s_1\ne s_2\in\sigma(V)\textrm{ with } j |s-s'|=|k||s_1-s_2|\},
$$
as mentioned before this set if finite because $|s_1-s_2|\le 2$.  We denote
$$
H=\Pi_{j\in\Lambda}\big(\big(L-R)^2-\mu_j^2\big)\big|_h,\qquad\qquad K=\Pi_{j\in\Lambda}\left(L+R-\mu_j\right)\big|_{sk}
$$
the restriction of the first operator to $(\mathcal A_0)_h$ (resp. the second one to $(\mathcal A_0)_{sk}$). Here $\mu_j=\frac{j}{|k|}|s-s'|$ as before. 
\end{rem}

\setstretch{1.25}
\begin{lem}\label{muu2} Let $L,R$ be as before, $\mu>0$. Let $z=w|z|\in\mathcal A_0$ be its polar decomposition in $\mathcal A_0''$. All the operators here are considered as operators in $\mathcal A_0$.
\begin{enumerate}
\item[a) ] $L+R+\mu 1$ is invertible, $(L-R)^2-\mu^2  1 $ is invertible for $\mu>1$.
\item[b) ] Let $\mathfrak H=\ker H$ and $\mathfrak K=\ker K$, then
\vspace*{-.2cm}$$\vspace*{-.15cm}
\qquad\mathfrak H=\oplus_{j\in\Lambda}\ker((L-R)^2-\mu_j^2)\big|_h\quad \textrm{ and }\quad\mathfrak K=\oplus_{j\in\Lambda}\ker(L+R-\mu_j)\big|_{sk}.
$$
\item[c) ] $\ker((L-R)^2-\mu^2)|_h=\{f+f^*+g+g^*: f\in\ker(L-R-\mu), g\in \ker(L-R+\mu)\}$.
\item[d) ]  $\ker(L-R\pm \mu)=\{z: |z||\lambda|=|\lambda||z|,\, w^*|\lambda|w=(|\lambda|\mp\mu)P_{|z|}\}$.
\item[e) ]  $\ker((L-R)^2-1)=\{0\}$.
\item[f) ] $\ker(L+R-\mu)=\{z:|z||\lambda|=|\lambda||z|,\, w^*|\lambda|w=(\mu-|\lambda|)P_{|z|}\}$.
\item[g) ]  $\ker(L+R-2)|_{sk}=\{b: \, |\lambda|b=b\}$.
\end{enumerate}
\end{lem}
\setstretch{1.1}
\begin{proof}
$\mathrm {a)}\,$ Since $\sigma(L+R+\mu)\subset [\mu,2+\mu]$ by  \cite[Theorem 10]{rosen}, then $L+R+\mu 1$ is always invertible; in the real case we complexify and use Remark \ref{complexifi} to arrive to the same conclusion. With the same approach  $(L-R)^2-\mu^2  1 $ is invertible for $\mu>1$, since $\sigma(L-R)^2\subset [0,1]$. 

$\mathrm {b)}\,$ We order the indices in $\Lambda=\{j_1,\dots,j_n\}$, and note that the $\mu_{j_k}$ are distinct. That eigenvalues corresponding to different eigenvectors are linearly independent, is well known and elementary. On the other hand by Bezout's identity we can find polyonmials $p_k$ such that
$$
\sum_{k=1}^n p_k(t)\prod_{l\ne k}(t-\mu_{j_l})=1.
$$
Replacing $t$ with $L+R$ we see that for any $a\in\mathcal A_0$ we have
$$
a=\sum_{k=1}^n p_k(L+R)\prod_{l\ne k}(L+R-\mu_{j_l})a=\sum_{k=1}^n a_k,
$$
and it is plain that $a_k \in \ker(L+R-\mu_{j_k})$. Thus the kernel of $K$ is the direct sum of the kernels of its factors. With a similar argument, now using that $(L-R)^2$ maps self-adjoint operators to self-adjoint operators, the kernel of $H$ is the direct sum of the kernels of its factors. 

$\mathrm {c)}\,$ If $f,g$ belong to the stated kernels then $f+g\in \ker((L-R)^2-\mu^2)$ and this operator preserves self-adjoint elements, thus the same holds for $f^*+g^*$, thus $f+g+f^*+g^*\in \ker((L-R)^2-\mu^2$. Now conversely, if $a=a^*\in \ker((L-R)^2-\mu^2)$, using the same argument of item $b)$ we can write $a=a_++a_-$ with $a_{\pm}\in \ker   (L-R\pm \mu)$ respectively. Take $f=(a_+^*+a_-)/4$, $g=(a_++a_-^*)/4$, one can check by hand that they belong to the stated kernels. Then
\begin{align*}
f+f^*+g+g^* &=1/4(a_+^*+a_-+a_++a_-^*+a_++a_-^*+a_+^*+a_-)\\
& =\mathfrak{Re}(a_++a_-)=\mathfrak{Re}(a)=a.
\end{align*}

$\mathrm {d)}\,$ If $z$ obeys both conditions, a direct substitution shows that $z\in \ker(L-R\pm \mu)$. Hence assumme $z\in \mathcal A_0$ verifies 
\begin{equation}\label{zkerr}
|\lambda| z -z|\lambda|=\mp \mu z
\end{equation}
then also $z^*|\lambda|- |\lambda|z^*=\mp \mu z^*$. We multiply the first equation on the left by $z^*$, the second one on the right by $z$ and we substract both equations to arrive to $ |z|^2 |\lambda|- |\lambda| |z|^2=0$, which proves  $|\lambda||z|=|z||\lambda|$.  With a similar approach one sees that $|\lambda||z^*|=|z^*||\lambda|$, hence $|\lambda| w |z|w^*=w|z|w^*|\lambda|$. We multiply this last one on the left by $w^*$ and with to $w$ on the right to obtain
\begin{equation}\label{wlambda}
w^*|\lambda| w |z|=|z|w^*|\lambda|w,
\end{equation}
recalling  $w^*w|z|=P_{|z|}|z|=|z|$. Now we multiply (\ref{zkerr}) on the left with $w^*$ to arrive to
$$
w^*|\lambda|w |z|- |z||\lambda|=\mp \mu |z|.
$$
Using (\ref{wlambda}) we can rewrite it as 
$$
|z| (w^*|\lambda|w -|\lambda|)=\mp \mu |z|,
$$
which implies $P_{|z|}(w^*|\lambda|w - |\lambda|)=\mp \mu P_{|z|}$. Since  $P_{|z|}w^*=w^*$ and  $P_{|z|}$ commutes with $|\lambda|$, we obtain $w^*|\lambda|w=(|\lambda|\mp\mu)P_{|z|}$. 

$\mathrm {e) }\,$ Let $z=w|z|\in \ker(L-R-1)$, then by $d)$ we have 
$$
P_{|z|}(|\lambda|+1)=w^*|\lambda|w\le w^*w=P_{|z|}
$$
hence $P_{|z|}|\lambda|\le 0$ which implies $|\lambda|P_{|z|}=0$. Since $P_{|\lambda|}=1$ is the identity of $\mathcal A_0$, the operator $|\lambda|$ has dense range, hence it is injective, thus it must be $P_{|z|}=0$ which implies $z=0$. Now if $z\in z=w|z|\in \ker(L-R+1)$ a similar reasoning (conjugating first with $w$) imples $z=0$ also. Hence 
$$
\ker ((L-R)^2-1)=\ker (L-R+1)\oplus \ker (L-R-1)=\{0\}.
$$

$\mathrm {f)}\,$ Has proof which is similar to that of $d)$ therefore it is omitted.

$\mathrm {g)}\,$  From the previous item we know that   $w|\lambda|w^*=P_b(2-|\lambda|)$ (here $P_b=P_{|b|}$ because $b^*=-b$). Now $\sigma(w|\lambda|w^*)\subset [0,1]$, and $\sigma(P_b(2-|\lambda|))\subset \{0\}\cup [1,2]$, since $|\lambda|$ and $P_b$ commute. Then it must be 
$$
\sigma(w|\lambda|w^*)=\sigma(P_b(2-|\lambda|))\subset \{0,1\}
$$
and this tells us that this operator is an orthogonal projection, i.e. $p=w|\lambda|w^*=P_b(2-|\lambda|)$. Multiplying by $P_b$ on the left we see that $P_bp=P_b(2-|\lambda|)=p$, and taking adjoint also $pP_b=p$. Thus $p=P_bp=P_bpP_b\le P_b^2=P_b$ and $p$ is a subprojection of $P_b$. We claim that $p=P_b$: from $p=2P_b-P_b|\lambda|$ we see that 
$$
0\le P_b-p=-(1-|\lambda|)P_b.
$$
Now the spectrum of the operator on the right is contained in $[-1,0]$ thus it must be $P_b-p=0$ as claimed. Then we have $P_b=2P_b-P_b|\lambda|$ from where we obtain $P_b=P_b|\lambda|$. Multiplying by $b$ on the left and taking adjoints we conclude that $b=b|\lambda|=|\lambda|b$ as claimed, and this finishes the proof.
\end{proof}

\begin{teo}[Monoconjugate points]\label{mononcon} 
Let $L=L_{|\lambda|}$ amd $R=R_{|\lambda|}$, let $T=T(k,s,s')$ and $\mathfrak H, \mathfrak K$ as in the previous lemma. Then
$$
\ker D(\Exp_P)_{TV}\big|_{P_v\mathcal A P_v}=\left\{ \left(\begin{array}{cc} 0 & \Omega (a+b) \\ (a-b)\Omega^* & 0 \end{array}\right)    : \,a\in\mathfrak H,\, b\in \mathfrak K\right\}.
$$
\end{teo}
\begin{proof}
Using the notation of Remark \ref{muu}, we have to find the kernel of $\Pi_-\oplus \Pi_+$. Then $X=\rchi+\rchi^*\in  \ker D(\Exp_P)_{TV}|_{P_v\mathcal A P_v}$ is obtained by taking $\rchi=\Omega z$, $\rchi^*=z^*\Omega^*$. Now  $\Pi_-$ restricted to the self-adjoint operators of $\mathcal A_0$ is exactly $H$, and  $\Pi_+$ restricted to the skew-adjoints of $\mathcal A_0$ (with the invertible factors removed) is exactly $K$. Then the assertion follows from the previous lemma.
\end{proof}

\smallskip

Now we show that the first conjugate point is always at $T=\pi/2$.

\begin{teo}[First conjugate point]\label{t1cp} The point $Q=\gamma(\frac{\pi}{2})$ is the first conjugate point to $P$ along $\gamma$. The kernel of $D(Exp_P)_{\frac{\pi}{2}V}$ in $T_PGr(P_0)$ 
$$
\{\Omega z-z\Omega^*: z\in \mathcal A_0, \,|\lambda|z=z=-z^*\}.
$$
If $Q$ is not monoconjugate, it is epiconjugate.
\end{teo}
\begin{proof}
At the first conjugate point we already noted (Remark \ref{ftcp}) that there is only one possible value of $\mu$, which is $\mu=2$, and in Remark \ref{pvc}.iv) we noted that the all the solutions lay in $P_v\mathcal A P_v$. By Lemma \ref{muu2}$.a)$ we have that $H$ is invertible hence $a=0$ and by the same lemma item $g)$ we see that the kernel of $K$ is exactly $b^*=-b$ such that $|\lambda|b=b$. Now we show that $L+R-2$ is never invertible in $\mathcal A_0$. Let $\varphi$ be a state of $\mathcal A_0$ such that $\varphi(|\lambda|)=\|\lambda\|=1$. Then $\varphi((L+R)1)=2\varphi(|\lambda|)=2$. Since $\varphi(1)=1=\|1\|$, we see that $2\in W(L+R)$, the spatial numerical range of the operator $L+R\in \mathcal B(\mathcal A_0)$. Now $W(L+R)\subset V(L+R)$, the later being the intrinsic numerical range of $L+R$ in the Banach algebra $\mathcal B(\mathcal A_0)$ (in fact, it holds $\overline{\mathrm{co}\,W(L+R)}=V(L+R)$, see \cite[p. 83]{bonsall}). Now note that 
$$
\|e^{is (L+R)}a\|=\|e^{is |\lambda|}ae^{is|\lambda|}\|=\|a\|
$$
for any $a\in\mathcal A_0$, hence $L+R$ is an Hermitian operator in $\mathcal B(\mathcal A_0)$. Since $L+R$ is Hermitian, it is well known that  $\mathrm{co}(\sigma(L+R))=V(L+R)$. Now by Remark \ref{speo}, $\sigma(L+R)\subset [t_0,2]$ for some $t_0\ge 0$. If $2$ does not belong to this spectrum then $\sigma(L+R)\subset [t_0,2-\varepsilon]$ for some $\varepsilon>0$. Then
$$
2\in W(L+R)\subset V(L+R)=\mathrm{co}(\sigma(L+R))\subset [t_0,2-\varepsilon],
$$
a contradiction. Hence $L+R-2$ is not invertible and we are done.
\end{proof}

\begin{rem}\label{qk}The points $Q_k=\gamma(\frac{k\pi}{2})$ are always conjugate to $P$: if they are not monoconjugate they are epiconjugate. This is because the factor $L+R-2$ is always present in $\sinhc(T\ad v)$ (by choosing $j=k$), and in the previous proof we showed that this operator is never invertible in $\mathcal A_0$.
\end{rem}

\begin{rem}[Counting dimensions] To compute the dimension of the kernel at the first tangent conjugate point, by the previous theorem it then  suffices to compute the dimension of $X=\{ \,b^*=-b  \,:\,|\lambda|b=b\}$. Since $\||\,\lambda|\,\|=1$, we know that $1\in\sigma(|\lambda|)$. But it is possible that $1$ is not an eigenvalue of $|\lambda|$, hence in that case the space $X$ is null, thus the kernel is null. 
\end{rem}

\smallskip

On the other hand, if $1$ is an eigenvalue of $V$ (equivalently, it is isolated in the spectrum), let $P_1\in\mathcal A_0$ be the associated eigenprojection let $d=\dim_{\mathbb R}(P_1)$,  it must be
$$
d\le \mathrm{rank}(P_{|\lambda|}) = 
\mathrm{rank}(|\lambda|)=\mathrm{rank}(\lambda^*\lambda)
\le \min\{ \mathrm{rank}(P),\mathrm{rank}(1-P)\}.
$$
Then for the complex Grassmannian we have that the order at $V_0=\frac{\pi}{2}V$ is $d^2$, and for the real Grassmanian it is $\frac{d^2-d}{2}$, since those numbers are the real dimensions of the spaces of skew-symmetric matrices acting on a space of real dimension $d$.

\begin{coro}\label{dimensiond} 
Let $d$ be the real dimension of the fixed point set of $V^2$. Then for the complex Grassmannian we have that the order at the first tangent conjugate point along $\gamma$ with initial unit speed $V$ is $d^2$, provided the partial isometry of $V=U|V|$ belongs to $\mathcal A$. For the real Grassmanian the number is $\frac{d^2-d}{2}$.
\end{coro}
\begin{proof}
Note that
$$
V^2=\left(\begin{array}{cc} \lambda\lambda^* & 0  \\ 0 & \lambda^*\lambda  \end{array}\right)=\left(\begin{array}{cc} |\lambda^*|^2 & 0  \\ 0 & |\lambda|^2  \end{array}\right).
$$
If we write $\xi=\xi_1+\xi_2\in \mathcal H=\ran(P)\oplus\ran(1-P)$, the fixed points of $V^2$ obey $|\lambda^*|^2\xi_1=\xi_1$ and $|\lambda|^2\xi_2=\xi_2$. Since $\Omega|\lambda|^2\Omega^*=|\lambda^*|^2$ (Remark \ref{polard}), the dimension of the fixed point set is the dimension of the eigenspace of the eigenvalue  $t=1$ of $|\lambda|^2$, or equivalently, of $|\lambda|$. By the previous discussion this number counts the dimension of the kernel.
\end{proof}

\begin{lem}[Eigenvalues of $V$ produce monoconjugate points]\label{otherconj}
Assumme $s\ne s'\in \sigma(V)$ are eigenvalues with respective eigenvectors $\xi_s,\xi_{'}$ and assume that both $\xi_s\otimes\xi_{'}$ and $\xi_{'}\otimes\xi_s$ belong to $\mathcal A$. Then $\gamma(T)$ for $T=T(k,s,s')$ is monoconjugate to $P=\gamma(0)$ for any $k\in\mathbb Z^*$.
\end{lem}
\begin{proof}
Consider the equations for $j=k$, $s_1=s$, $s_2=s'$ which appear in Remark \ref{muu}
\begin{align}
|\lambda|^2a+a|\lambda|^2-2|\lambda|a|\lambda|-|s-s'|^2a &=0\label{B1}\\
|\lambda|^2b+b|\lambda|^2+2|\lambda|b|\lambda|-|s-s'|^2b &=0\label{B2}.
\end{align}
for self-adjoint $a\in \mathcal A_0$ (resp. skew-adjoint $b$). Assumme $\mathcal A_0$ represented in some Hilbert space $\mathcal H$. We can safely assume that $0<s\le 1$. Consider first the case of $0< s'<s$. Let $\xi_s,\xi_{'}\in \mathcal H$ such that $|\lambda|\xi_s=s\xi_s$ and $|\lambda|\xi_{'}=s'\xi_{'}$. Note that $\xi_s,\xi_{'}\in \ran P_{|\lambda|}\cap\ran(1-P)$ since $\xi_s=s^{-1}|\lambda|\xi_s$ and likewise for $s'$. Consider $a=a^*=\xi_s\otimes\xi_{'}+\xi_{'}\otimes \xi_s\in\mathcal A$, then 
$$
P_{|\lambda|}a=(P_{|\lambda|}\xi_s)\otimes\xi_{'}+(P_{|\lambda|}\xi_{'})\otimes \xi_s=a
$$
and likewise $aP_{|\lambda|}=a$. The same argument shows that $(1-P)a=a=a(1-P)$, which shows that $a\in \mathcal A_0$. On the other hand, it is easy to  check that $a$ is a solution of equation (\ref{B1}), hence the kernel is nontrivial. If $s<s'\le 1$, the same solution applies since we can exchange $s,s'$. Now assume  $s'\le 0$, take $\xi_s$ as before and $\xi_{'}$ such that $|\lambda|\xi_{'}=|s'|\xi_{'}=-s'\xi_{'}$. In this case consider $-b^*=b=\xi_s\otimes\xi_{'}-\xi_{'}\otimes \xi_s\in\mathcal A_0$. Then it is easy to check that $b$ is a solution of equation (\ref{B2}). If $s'\ne -s$, then $b\ne 0$ and the kernel is nontrivial. If $s'=-s$, one would need to ask that the eigenspace of $s$ has real dimension at least $2$ (this is plain for the complex case), for in this case one can take two linearly independent eigenvectors $\xi_s,\xi_{'}$ of the eigenvalue $s$, and then $b\ne 0$.
\end{proof}

\begin{rem}[Compact operators and the restricted Grassmannian]\label{compact} Consider the unitary group of a proper ideal $\mathcal I\subset\mathcal K(\mathcal H)$ of compact operators (cf. Gohberg and Krein \cite[Chapter III]{gk}):
$$
\mathcal U_{\mathcal I}=\{u\in\mathcal U(H):u-1\in \mathcal I\}=\exp\{A: A^*=-A\in\mathcal I\}.
$$
A relevant case of infinite dimensional Grassmannian occurs when we consider the coadjoint orbit a a projection $P\in\mathcal B(\mathcal H)$ for the action of the group $\mathcal U_{\mathcal I}$, i.e.
$$
Gr_{\mathcal I}(P_0)=\{uPu^*: u\in\mathcal U_{\mathcal I}\}.
$$ 
Then the number $d$ of Corollary \ref{dimensiond} is finite, and the statement of that corollary holds for the restricted Grassmannian, with the same proof (despite the fact that in general it is not the unitary group of a $C^*$-algebra).

\smallskip

We also want to mention that the  argument in the previous Lemma \ref{otherconj} holds for these restricted Grassmannians (since the ideal $\mathcal I$ contains all finite rank operators $\xi_s\otimes\xi_s'$), with one exception. Indeed since $V$ is compact then $|\lambda|$ is positive compact, thus for each nonzero eigenvalue we have a finite dimensional nontrivial eigenspace. The exception is for the case of $s\ne 0, s'=0$, i.e. the candidate to conjugate point $v$ such that $\ad^2v+s^21=0$ with $0\ne s\in\sigma(|\lambda|)$. This is an exception because the kernel of $V$ (equivalently, of $|\lambda|$) might be trivial thus we cannot build neither $a$ nor $b$. Thus \textit{all candidates $Q=\gamma(T)$ are monoconjugate to $P=\gamma(0)$}, except perhaps for the case of $T=(k,s,0)$. Other restricted Grassmannians can be approached with our techniques, for instance those considered in \cite{ratiu} by Ratiu et al.
\end{rem}
 
\begin{ejem}[First epiconjugate point which is not monoconjugate]\label{noesmono}  Let $\mathcal H=L^2[-1,1]$, and let $\mathcal A=\mathcal B(\mathcal H)$. Let $P$ be the orthogonal projection given taking the even part of a function $f\in \mathcal H$, i.e.
$$
Pf(x)=\frac{1}{2}(f(x)+f(-x)).
$$
Let $V^*=V\in\mathcal A$ be given by $Vf(x)=xf(x)$. Then $\sigma(V)=[-1,1]$ and in particular $\|V\|=1$. Moreover if $v=[V,P]$ then  $vf(x)=xf(-x)$. Now
$$
PVf(x)=\frac{1}{2}(xf(x)-xf(-x))=V(1-P)f(x)
$$
hence $V=PV+VP$ thus $V$ is $P$-codiagonal. Let $\gamma$ be the geodesic through $P$ with intial speed $V$. We claim that $Q=\gamma(\frac{\pi}{2})$ \textit{is not monoconjugate but epiconjugate to $P$}.

\smallskip

To this end, note first that $|V|f(x)=|x|f(x)$ and it is plain that $\sigma(|V|)=[0,1]$ while $|V|$ has no eigenvalues. This also tells us that  $P_Vf(x)=f(x)$ i.e. $P_V$ is the identity operator. Therefore $\sinhc(T\ad V)$ is unitary equivalent to $H\oplus K$ for any $T$ (Remark \ref{pvc}.i). We have $P_{|\lambda|}=1-P$ and moreover $|\lambda|=(1-P)|V|(1-P)$ is described by $(|\lambda|f)(x)=|x|f(x)$ for odd functions $f\in L^2[-1,1]$, which is the range of $1-P$. We use the characterization of first conjugate points obtained in Theorem \ref{t1cp}.  The point $Q$ must be conjugate to $P$ by Theorem \ref{t1cp}. But $|\lambda|b=b$ for $b^*=-b$ has no solutions in $\mathcal A_0$, because $|\lambda|$ has no eigenvalues. Hence the point is not monoconjugate but epinconjugate.\end{ejem}

\smallskip

By taking the product of algebras, we show an  example where 

\begin{ejem}[$Q=\gamma(\pi/2)$ is monoconjugate and epiconjugate to $P$]\label{eslasdos} Consider the direct sum $\mathcal A=\mathcal B(\mathcal H)\oplus M_2(\mathbb C)$, with the maximum norm, with $\mathcal H=L^2[-1,1]$ as above. Let $P,V$ be as in the previous example and let $p=e_1\otimes e_1\in M_2(\mathbb C)$, while $w=e_1\otimes e_2-e_2\otimes e_1$ is $p$-codiagonal, skew adjoint and of unit norm. consider $P'=(P,p)$, $V'=(V,w)$, then with the product in each coordinate it is plain that $P'$ is a projection and $V'$ is $P'$-codiagonal and of unit norm. In both cases $P_v,P_w$ is the unit of the respective algebra, hence $P_{V'}$ is the unit of $\mathcal A$. Therefore for the first conjugate point we are again dealing only with the right-down corner of the algebra $\mathcal A$ by Remark \ref{pvc}.i), which is the direct sum of both corners. Moreover $|\lambda|(f,\xi)=(|x|f,\xi)$ for odd $f\in L^2[-1,1]$ and $\xi\in \mathbb C$, i.e.
$$
L_{|\lambda|}=\left(\begin{array}{cc}
M & 0\\
0 & 1
\end{array}\right)
$$
Where $M=M_{|x|}$ is the multiplication operator. In the second coordinate, the kernel is the span of $b=(0,i)$, in particular $Q$ is monoconjugate to $P$. Now $L_M+R_M-2$ is not invertible and injective (previous example), therefore it is not surjective and $Q$ is also epiconjugate to $P$.
\end{ejem}

\subsection{Projective spaces}\label{spc} We now characterize the kernel for all conjugate points in projective spaces, presented as the orbit of a one-dimensional projection:

\begin{ejem}[Complex projective space]\label{3por3} Here $P$ is (complex) one-dimensional projection. In this case $\lambda\rchi^*$, $\rchi\lambda^*$ are complex numbers and $\lambda\lambda^*$ is a real non-negative number. The normalization condition implies that 
$\lambda\lambda^*=\|\lambda\lambda^*\|=1$, and this also tells us that $p=|\lambda|^2=\lambda^*\lambda$ is a one-dimensional projection in $\mathcal A_0$. It is apparent that $\Omega=\lambda\in\mathcal A$ and in particular solving for $z\in \mathcal A_0$ solves the problem in $\mathcal A$.
\end{ejem}

\begin{prop}\label{losproye}
For each tangent $V$ in the complex projective space, there are two kinds of monoconjugate points: 
\begin{enumerate}
\item[i)] $T_0=(2k+1)\frac{\pi}{2}, \, k\in\mathbb Z$: the kernel is spanned by $X=iV$ (in particular  the order is always $1$)
\item[ii)] $T_1=k\pi, k\in \mathbb Z_{\ne 0}$: the kernel is given by all $b^*=-b\in\mathcal A_0$, and the whole $\widetilde{\mathcal C}_{P_v}$.
\end{enumerate}
For the case of $Gr(P_0)=\mathbb C^n/\mathbb C$, the order of the $T_1$ points is $2(2n-3)$.
\end{prop}
\begin{proof}
Since $\sigma(V)=\pm \sigma(|\lambda|)=\{-1,0,1\}$ we have only two possibilities $\{|s-s'|:s\ne s'\in\sigma(V)\}=\{1,2\}$. From Lemma \ref{fcp} we then obtain two families of times
$$
T_0=(2k+1)\frac{\pi}{2} \quad (k\in\mathbb Z)\quad \textrm{ and } \quad T_1=k\pi\quad (k\in \mathbb Z_{\ne 0}),
$$
the first one corresponding to $|s-s'|=2$ and the second one corresponding to $|s-s'|=1$. For the points of type $T_0$, we note that $2j=(2k+1).s_0$ is impossible because here the only possibility for $0\ne s_0\in\sigma(|V|)$ is $s_0=1$, thus by Remark \ref{pvc}.iii) the $P_v$-codiagonal part of the kernel is null. Now we examine the other part: by Theorem \ref{mononcon}  it must be $\frac{2|j|}{|2k+1|}=1$ or $\frac{2|j|}{|2k+1|}=2$. The first possibility is again excluded since $j\ne 0$ is integer. The second possibility occurs when $|j|=|2k+1|$, and the only operator (\ref{opermu}) that might not be invertible is for $\mu=2$. By Lemma  \ref{muu2} we have $a=0$ (since $\mu>1$)  by the same lemma $|\lambda|b=b$. Since $|\lambda|$ is a (real) $2$-dimensional projection the unique skew-adjoint solution is $b=i|\lambda|=i\lambda^*\lambda$ or equivalently $\rchi=\Omega b=\lambda i\lambda^*\lambda=i\lambda$, thus $X=iV$. 

Now we consider the points of the second kind $T_1=k\pi$. We first take a look in $\mathcal A_0$: by Theorem \ref{mononcon} we have to consider the cases $|j|=|k|$ or $|j|=|2k|$ which are both possible, hence we have $\mu=1$ and $\mu=2$, i.e.
\begin{align*}
((L-R)^2- 1)((L-R)^2-4)a=0\\
((L+R)^2- 1)((L+R)^2-4)b=0.
\end{align*}
By Lemma \ref{muu2}$.a)$ and $e)$ we have $a=0$. Cancelling the invertible terms in the second equation (Lemma \ref{muu2}), it must be $((L+R)- 1)((L+R)-2)b=0$. Since in this case $|\lambda|=P_{|\lambda|})$, we have 
$$
(L+R-2)b=pb+bp-2b=b+b-2b=0
$$
for any $b^*=-b\in\mathcal A_0$. For the case of $\mathcal A=M_n(\mathbb C)$, the dimension of $\ran(1-P)\mathcal A(1-P)$ is $n-1$, and inside it the dimension of $p=P_{|\lambda|}$ is $1$. Thus the dimension of $\mathcal A_0$ is $n-2$. Thus the space of skew-adjoint operators has complex dimension $n-2$, and real dimension $2n-4$. Now we take a look in $\mathcal C_{P_v}$, we note that it must be $j.1=k.1$ which is possible, hence by Remark  \ref{pvc}.iii) the kernel is non-trivial. Note that $|V|=P_v$ hence equation (\ref{sxpv}) gives
$$
xP_v=\sinhc(k\pi P_v)x=(1-P_v)x
$$
showing that any $x\in\widetilde{\mathcal C}_{P_v}$ is a solution. For the case of $\mathbb C^n/\mathbb C$, the complex dimension of $\widetilde{\mathcal C}_{P_v}$ is $n-1$, hence the real dimension is $2n-2$. Adding $2n-4+2n-2$ gives the full dimension of the kernel.
\end{proof}

\begin{ejem}[Real projective space]\label{3por3R} In this case the computations are done in the same fashion as in the previous example, but now the subspaces are real. Therefore the solutions for the points of type $T_0$ is $\{0\}$, because there are no skew-adjoint matrix in real dimension $1$. Therefore these are not monoconjugate points. The points of type $T_1$  fulfill the same conditions as in the previous example: $b^*=-b\in \mathcal A_0$  and the whole space of $P_v$ co-diagonal operators. For the case of $Gr(P_0)=\mathbb R^n/\mathbb R$, these spaces have dimension $n-2$ and $n-1$ respectively so the order of these points is $2n-3$.
\end{ejem}

\subsection{Beyond first conjugate point}\label{sbc}

By means of the calculations and remarks of the previous sections (in particular Remark \ref{muu}), the study of the differential of the exponential map $\Exp_P$ along $\gamma(t)=\Exp_p(tV)$, restricted to the subspace $P_vT_PGr(P_0)P_v$ can be subsumed to the study of the pair of operators $H,K\in \mathcal B(\mathcal A_0)$ defined before Lemma \ref{muu2}. Since $H,K$ preserve the spaces of self-adjoint and skew adjoint operators, it is plain that $D(\Exp_P)$ is not surjective here (resp. not bounded below) if and only if either $H$ or $K$ is not surjective (resp. not bounded below) here. Recall (Lemma \ref{fcp}) that
$$
T(k,s,s')=\frac{k\pi}{|s-s'|},\qquad k\in\mathbb Z^*=\mathbb Z\setminus\{ 0\}, \quad s\ne s'\in\sigma(V)
$$
are the candidates to conjugate points along the geodesic $\gamma(t)=\Exp_P(tV)$, for $V\in T_PGr(P_0)$ of unit length. The points $Q_k=\gamma(\frac{k\pi}{2})$ are always conjugate to $P$ along $\gamma$ (Remark \ref{qk}). For other $T=T(k,s,s')$, for some particular algebras this is also true.

\begin{lem}\label{prime}
Assumme $\mathcal A_0$ is a von Neumann factor or a prime $C^*$-algebra. Take any $T=T(k,s,s')$. If $Q=\gamma(T)$ is not monoconjugate to $P$, then $Q$ is epiconjugate to $P$.
\end{lem}
\begin{proof}
The case of $j=|k|$ is always possible, hence we have that at any $T=T(k,s,s')$ the operator $K$ contains the factor $L+R-\mu_k=L+R-|s-s'|$.  Now recall $\mathcal A_0$ is a $C^*$-algebra with identity (Remark \ref{muu}). If $\mathcal A_0$ is a von Neumann factor, or a  prime $C^*$-algebra, we have the equality of spectra
$$
\sigma(L\pm R)=\{|s_1|\pm |s_2|: s_1,s_2\in \sigma(V)\}
$$
(see Remark \ref{speo}). We can safely assumme that $s\ge s'$. If $ss'\ge 0$ and both are non-negative then $|s-s'|=s-s'=|s|-|s'|\in \sigma (L-R)$ thus $L-R-|s-s'|$ is not invertible in $\mathcal A_0$. Then it is not invertible in $(\mathcal A_0)_h$ (Remark \ref{complexifi}). Likewise, if both $s,s'$ are non-positive then again $|s-s'|=s'-s=|s|-|s'|\in \sigma(L-R)$. Further, if $ss'<0$ then $|s-s'|=s-s'=|s|+|s'|\in\sigma(L+R)$ and $L+R-|s-s'|$ is not invertible in $(\mathcal A_0)_{sk}$. Thus in any case $H$ or $K$ (or both) is not invertible. If $H,K$ are injective, all of their factors must be injective; since they are not invertible, they cannot be bounded below. Now $\|A\| \|B\xi\|\ge \|AB\xi\|\ge c\|\xi\|$ implies $B$ bounded below; since the factors in $H,K$ commute, either $H$ or $K$ are not bounded below. Since $H,K$ are a direct summands of $\sinhc(T \ad v)$, the later cannot be bounded below, and its range cannot be full. 
\end{proof}

We close the paper with a finite dimensional example where $Q=\gamma(T)$ for $T(k,s,s')$ \textit{is not} conjugate to $P$, except for the case of the points $T=\frac{k\pi}{2}$ already discussed in Remark \ref{qk}. This is unlike the classical Grassmannians $Gr_k(n)$ where they all are conjugate, and the main reason of failure is that $\mathcal A_0\simeq \mathbb C\oplus \mathbb C$ is not a factor.

\begin{ejem}\label{pocos} Let $\mathcal A=M_2(\mathbb C)\oplus M_2(\mathbb C)$, let
$$
P=\left(\begin{array}{cc} 1 & 0 \\ 0 & 0 \end{array}\right)\oplus \left(\begin{array}{cc} 1 & 0 \\ 0 & 0 \end{array}\right)\qquad \mathrm{ and }\quad V=\left(\begin{array}{cc} 0 & 1 \\ 1 & 0 \end{array}\right)\oplus \left(\begin{array}{cc} 0 & \alpha \\ \alpha & 0 \end{array}\right)
$$
for some $0<\alpha<1$. Then $V$ is $P$-codiagonal, $\sigma(V)=\{-1,-\alpha,\alpha,1\}$ and $P_V=1\oplus 1$ is the identity of $\mathcal A$. We also have 
$$
\mathcal A_0=\left(\begin{array}{cc} 0 & 0 \\ 0 & \mathbb C  \end{array}\right)\oplus \left(\begin{array}{cc} 0 & 0 \\ 0 & \mathbb C \end{array}\right), \quad |\lambda|=\left(\begin{array}{cc} 0 & 0 \\ 0 & 1  \end{array}\right)\oplus \left(\begin{array}{cc} 0 & 0 \\ 0 & \alpha \end{array}\right), 
$$
and the identity of $\mathcal A_0$ is of course $1-P$. Then $\sigma(|\lambda|)=\{1,\alpha\}$ but $L-R=0$ in $\mathcal A_0$ hence $\sigma(L-R)\subsetneq\{0,1-\alpha,\alpha-1\}$. On the other hand $L+R=2 \oplus 2\alpha$ in $\mathcal A_0$ hence $\sigma(L+R)=\{2,2\alpha\}$ again with strict inclusion in $\{2\alpha, 1+\alpha, 2\}$. There are four family of candidates to conjugate points,
$$
T_1=\frac{k\pi}{2},\quad T_2=\frac{k\pi}{1+\alpha},\quad T_3=\frac{k\pi}{1-\alpha}, \quad T_4=\frac{k\pi}{2\alpha}.
$$
For the first family we know that $\gamma(T_1)$ is conjugate to $P$ (Remark \ref{qk}), in fact monoconjugate because the algebra is finite dimensional. On the other hand it is easy to see that \textit{none of the other points are conjugate to $P$}: we only show that for the case of $T_2$, the other cases being similar. For this case one can check that the only possible value of $\mu_j$ is $\mu=1+\alpha>1$. Since $P_V=1$ all conjugate points occur inside $\mathcal A_0$ (Remark \ref{pvc}$.1$). Therefore we are only interested in 
$$
H=(L-R)^2-(\alpha +1)=-(\alpha+1)
$$
which is invertible and 
$$
K=L+R-(\alpha+1)=(1-\alpha)\oplus (\alpha-1)
$$
which is also invertible. Hence $\sinhc(T_2\ad v)$ is invertible and $\gamma(T_2)$ is not conjugate to $P$ along $\gamma$.
\end{ejem}

\end{document}